\documentclass[journal]{IEEEtran}
\usepackage{amsfonts}
\usepackage{amsmath}
\usepackage{amssymb}
\usepackage{amsthm}
\usepackage{etoolbox}
\usepackage{fullpage}
\usepackage{floatrow}
\usepackage{graphicx}
\usepackage{epstopdf}
\usepackage{mathtools}
\usepackage{caption}
\usepackage{subcaption}
\usepackage{multirow}
\usepackage{lipsum}

\def\mllap{\mathpalette\mllapinternal}
\def\mllapinternal#1#2{\llap{$\mathsurround=0pt#1{#2}$}}

\makeatletter
\newcommand{\pushright}[1]{\ifmeasuring@#1\else\omit\hfill$\displaystyle#1$\fi\ignorespaces}
\newcommand{\pushleft}[1]{\ifmeasuring@#1\else\omit$\displaystyle#1$\hfill\fi\ignorespaces}
\makeatother
\newcommand{\oo}{\mathit{o}}
\newcommand{\OO}{\mathcal{O}}

\makeatletter
\g@addto@macro\th@plain{\thm@headpunct{\ }}
\makeatother
\makeatletter
\g@addto@macro\th@definition{\thm@headpunct{\ }}
\makeatother
\makeatletter
\def\maketag@@@#1{\hbox{\m@th\normalfont\normalsize#1}}
\makeatother

\newcommand{\definition}[1]{\textbf{#1}}

\csundef{proof}
\csundef{endproof}
\theoremstyle{plain}\newtheorem{theorem}{Theorem}
\theoremstyle{plain}\newtheorem{corollary}{Corollary}
\theoremstyle{definition}\newtheorem*{proof}{Proof}
\theoremstyle{definition}\newtheorem{remark}{Remark}

\graphicspath{{}}

\begin{document}

\title{Spectral Statistics of Lattice Graph\\ Percolation Models}
\author{Stephen Kruzick and Jos\'{e} M. F. Moura,~\IEEEmembership{Fellow,~IEEE}  % <-this % stops a space
\thanks{The authors are with the Department of Electrical and Computer Engineering, Carnegie Mellon University, Pittsburgh, PA 15213 USA (ph: (412)2686341; e-mail: skruzick@andrew.cmu.edu; moura@ece.cmu.edu).} % <-this % stops a %space
\thanks{This work was supported by NSF grant \#~CCF1513936.}
}
\maketitle

\begin{abstract}
In graph signal processing, the graph adjacency matrix or the graph Laplacian commonly define the shift operator. The spectral decomposition of the shift operator plays an important role in that the eigenvalues represent frequencies and the eigenvectors provide a spectral basis. This is useful, for example, in the design of filters.  However, the graph or network may be uncertain due to stochastic influences in construction and maintenance, and, under such conditions, the eigenvalues of the shift matrix become random variables.  This paper examines the spectral distribution of the eigenvalues of random networks formed by including each link of a $D$-dimensional lattice supergraph independently with identical probability, a percolation model.  Using the stochastic canonical equation methods developed by Girko for symmetric matrices with independent upper triangular entries, a deterministic distribution is found that asymptotically approximates the empirical spectral distribution of the scaled adjacency matrix for a model with arbitrary parameters.  The main results characterize the form of the solution to an important system of equations that leads to this deterministic distribution function and significantly reduce the number of equations that must be solved to find the solution for a given set of model parameters.  Simulations comparing the expected empirical spectral distributions and the computed deterministic distributions are provided for sample parameters.
\end{abstract}


\section{Introduction}

Much of the increasingly vast amount of data available in the modern day exhibits nontrivial underlying structure that does not fit within classical notions of signal processing.  The theory of graph signal processing has been proposed for treating data with relationships and interactions best described by complex networks.  Within this framework, signals manifest as functions on the nodes of a network.  The shift operator used to analyze these signals is provided by a matrix related to the network structure, such as the adjacency matrix, Laplacian matrix, or normalized versions thereof \cite{ASan1}\cite{DShu1}.  Decomposition of a signal according to a basis of eigenvectors of the shift operator serves a role similar to that of the Fourier Transform in classical signal processing \cite{ASan2}.  In this context, multiplication by polynomial functions of the chosen shift operator matrix performs shift invariant filtering \cite{ASan1}.

The eigenvalues of the shift operator matrix play the role of graph frequencies and are important in the design and analysis of graph signals and graph filters.  If $W$ is a diagonalizable graph shift operator matrix with eigenvalue $\lambda$ for eigenvector $\mathbf{v}$ such that $W\mathbf{v}=\lambda \mathbf{v}$ and, for example, if a filter is implemented on the network as $P\left(W\right)$ where $P$ is a polynomial, then $P\left(W\right)$ has corresponding eigenvalue $P\left(\lambda\right)$ for $\mathbf{v}$ by simultaneous diagonalizability of powers of $W$ \cite{ASan4}.  The framework of graph signal processing regards $P\left(\lambda\right)$ as the frequency response of the filter \cite{ASan3}.  Hence, knowledge of the eigenvalues of $W$ informs the design of the filter $P$ when the eigenvalues of $P\left(W\right)$ should satisfy desired properties.  Furthermore, the eigenvalues relate to a notion of signal complexity known as the signal total variation, which has several slightly different definitions depending on context \cite{DShu1}\cite{ASan2}\cite{ASan3}\cite{SChe1}.  For purposes of motivation, taking the shift operator to be the row-normalized adjacency matrix $\widehat{A}$, define the $l_p$ total variation of the signal $\mathbf{x}$ as
\begin{equation}\label{TotalVariation}
\operatorname{TV}_{\mathcal{G}}\left(\mathbf{x}\right)=\left\|\left(I-\widehat{A}\right)\mathbf{x}\right\|_p^p=\left\|\widehat{L}\mathbf{x}\right\|_p^p
\end{equation}
 which sums over all network nodes the $p$th power of the absolute difference between the value of the signal $\mathbf{x}$ at each node and the average of the value of $\mathbf{x}$ at neighboring nodes \cite{ASan3}\cite{SChe1}.  Thus, if $\mathbf{v}$ is a normalized eigenvector of the row-normalized Laplacian $\widehat{L}$ with eigenvalue $\lambda$, $\mathbf{v}$ has total variation $|\lambda|^p$.  The eigenvectors that have higher total variation can be viewed as more complex signal components in much the same way that classical signal processing views higher frequency complex exponentials.

As an application, consider a connected, undirected network on $N$ nodes with normalized Laplacian $\widehat{L}$ with the goal of achieving distributed average consensus via a graph filter.  For such a network, there is a simple Laplacian eigenvalue $\lambda_1(\widehat{L})=0$ corresponding to the averaging eigenvector $\mathbf{v}_1=\mathbf{1}$ and other eigenvalues $0<\lambda_i(\widehat{L})\leq2$ for $2\leq i \leq N$.   Any filter $P$ such that $P\left(0\right)=1$ and $|P(\lambda_i(\widehat{L}))|<1$ for $2\leq i \leq N$ will asymptotically transform an initial signal $\mathbf{x}_0$ with average $\mu$ to the average consensus signal $\mathbf{\bar{x}}=\mu \mathbf{1}$ upon iterative application \cite{EKok1}.  If the eigenvalues $\lambda_i(\widehat{L})$ for $2\leq i \leq N$ are known, consensus can be achieved in finite time by selecting $P$ to be the unique polynomial of degree $N-1$ with $P\left(0\right)=1$ and $P(\lambda_i(\widehat{L}))=0$ \cite{ASan4}.  Note that the averaging eigenvector has total variation $0$ by the above definition \eqref{TotalVariation} and that all other, more complex eigenvectors are completely removed by the filter.  Thus, a finite time consensus filter represents the most extreme version of a non-trivial lowpass filter.  With polynomial filters of smaller fixed degree $d$, knowledge of the eigenvalues can be used to design filters of a given length for optimal consensus convergence rate, which is given by
\begin{equation}
\rho\left(P\left(\widehat{L}\right)\right)^{1/d}=\max\limits_{2\leq i \leq N}\left|P\left(\lambda_i\left(\widehat{L}\right)\right)\right|^{1/d}
\end{equation}
as studied in \cite{EKok1}.  This can also be attempted in situations in which the graph is a random variable, leading to uncertainty in the eigenvalues.  For example, \cite{EKok1} also proposes two methods to accomplish this for certain random switching network models.  One relies on the eigenvalues of the expected iteration matrix.  The other attempts to filter over a wide range of possible eigenvalues without taking the model into account.  Neither takes into account any information about how the eigenvalue distribution spreads.  This type of information could be relevant to graph filter design, especially when the network does not switch too often relative to the filter length, rendering static random analysis applicable.

The empirical spectral distribution $F_W\left(x\right)$ of a Hermitian matrix $W$, which counts the fraction of eigenvalues of $W$ in the interval $\left(-\infty,x\right]$, describes the set of eigenvalues \cite{VGir1}.  For matrices related to random graphs, the empirical spectral distributions are, themselves, function-valued random variables.  Often, it is not possible to know the joint distribution of the eigenvalues, but the behavior of the empirical spectral distribution can be described asymptotically.  For instance, the empirical spectral distributions of Wigner matrices converge to the integral of a semicircle asymptotically in the matrix size as described by Wigner's semicircle law \cite{EWig1}.  In other cases, a limiting distribution may or may not exist, but a sequence of deterministic functions can still approximate the empirical spectral distribution \cite{RCou1}.  The stochastic canonical equation techniques of Girko detailed in \cite{VGir1}, which allow analysis of matrices with independent but non-identically distributed elements as necessary when studying the adjacency matrices of many random graph models, provide a method to obtain such deterministic equivalents.

\begin{figure}[t]
\includegraphics[width=.8\textwidth]{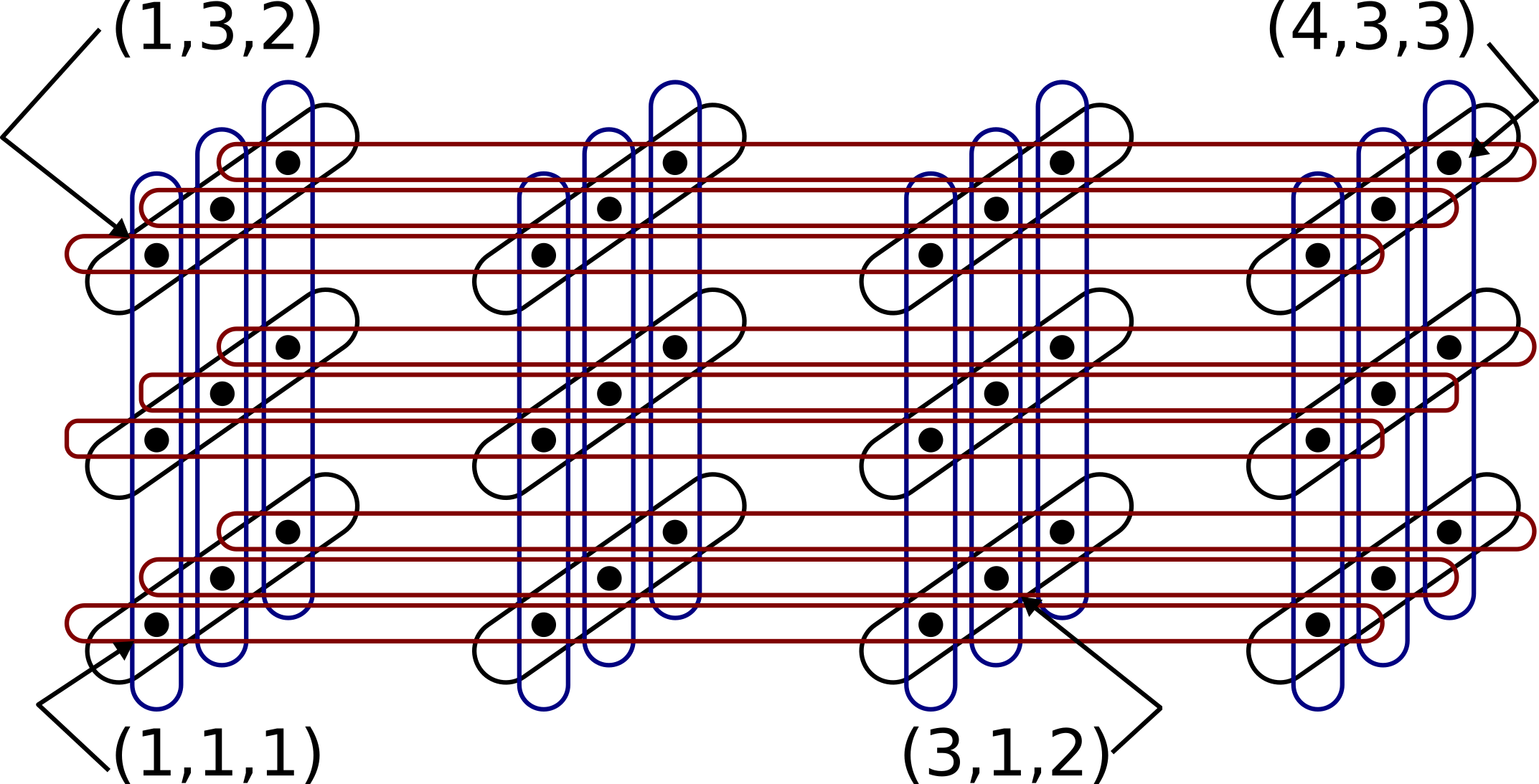}
\caption{The illustration shows an example lattice graph with three dimensions of size $4 \times 3 \times 3$ with some node tuples labeled.  Each group of circled nodes, which differ by exactly one symbol, represents a complete subgraph.}
\label{ExLat}
\end{figure}

This paper analyzes the empirical eigenvalue distributions of the adjacency matrices of random graphs formed by independent random inclusion or exclusion of each link in a certain supergraph, a bond percolation model \cite{GGri1}.  Specifically, the paper examines adjacency matrices of percolation models of $D$-dimensional lattice graphs, in which each $D$-tuple with $M_d$ possible symbols in the $d$th dimension has an associated graph node and in which two nodes are connected by a link if the corresponding $D$-tuples differ by only one symbol.  These graphs generalize other commonly encountered graphs, such as complete graphs and the cubic lattice graph \cite{RLas1}, to an arbitrary lattice dimension.  An illustration for the three dimensional case appears in Figure \ref{ExLat}, in which a complete graph connects each group of circled nodes.  The stochastic canonical equation techniques of Girko provide the key mathematical tool employed here to achieve our results.  The resulting deterministic approximations to the adjacency matrix empirical spectral distribution can, furthermore, provide information about the empirical spectral distributions of the normalized adjacency matrices and, thus, the normalized Laplacians.   This work is similar to that of \cite{KAvr1}, which analyzes the adjacency matrices of a different random graph model known as the stochastic block model using similar tools and suggests applications related to the study of epidemics and algorithms based on random walks.

Section \ref{SecBackground} discusses relevant background information on random graphs and the spectral statistics of random matrices.  It also introduces an important mathematical result from \cite{VGir1}, presented in Therorem~\ref{GirkoK01Thm}, which can be used to find deterministic functions that approximate the empirical spectral distribution of certain large random matrices by relating it to the solution of a system of equations \eqref{GirkoK01Eq}.  Section \ref{SecMainResults} introduces lattice graphs and addresses application of Theorem \ref{GirkoK01Thm} to scaled adjacency matrices of percolation models based on arbitrary $D$-dimensional lattice graphs.  The main contributions of this paper appear in Theorem \ref{MainThm}, which derives the form of the solution to the system of equations \eqref{GirkoK01Eq} for arbitrary parameters, and also in Corollary \ref{MainCor}, which describes the particular solution for given model parameters.  The relationship of the empirical spectral distribution of the adjacency matrix to those of the symmetrically normalized adjacency matrix and the row-normalized adjacency matrix is also noted.  For some example lattice graphs, these results are used to provide deterministically calculated distribution functions that are compared against the expected empirical spectral distributions.  Finally, Section \ref{SecConclusion} provides concluding analysis.

\section{Background}\label{SecBackground}

\subsection{Random Graphs}

An \definition{undirected graph (network)} $\mathcal{G}$ consists of an ordered pair $\left(\mathcal{V},\mathcal{E}\right)$ in which $\mathcal{V}$ denotes a set of \definition{nodes (vertices)} and $\mathcal{E}$ denotes a set of \definition{undirected links (edges)}.  These undirected links are unordered pairs $\left\{v_i,v_j\right\}=\left\{v_j,v_i\right\}$ of nodes $v_i,v_j\in \mathcal{V}$ with $v_i\neq v_j$.  Note that \definition{self-loops}, links formed as $\left\{v,v\right\}$, are excluded from this definition.  The graph \definition{adjacency matrix} $A\left(\mathcal{G}\right)$ encapsulates the graph information with $A\left(\mathcal{G}\right)_{ij}=1$ if $\left\{v_j,v_i\right\}\in \mathcal{E}$ and $A\left(\mathcal{G}\right)_{ij}=0$ if $\left\{v_j,v_i\right\}\notin \mathcal{E}$.  A \definition{subgraph} $\mathcal{G}_{\mathrm{sub}}$ of an undirected graph is a graph $\left(\mathcal{V}_{\mathrm{sub}},\mathcal{E}_{\mathrm{sub}}\right)$ with $\mathcal{V}_{\mathrm{sub}}\subseteq \mathcal{V}$ and $\mathcal{E}_{\mathrm{sub}}\subseteq \mathcal{E}$, and $\mathcal{G}$ is said to be a \definition{supergraph} of $\mathcal{G}_{\mathrm{sub}}$.  \definition{Random undirected graphs}, which are undirected graph-valued random variables, model uncertainty in the network topology.  There are many commonly studied random graph models, which may be specified by a probability distribution on the collection of possible link sets $P\left(\mathcal{V}\times \mathcal{V}\right)$ between a given number of nodes $\left|\mathcal{V}\right|=N$.

One of the most frequently considered random graph models, the \definition{Erd\"{o}s-R\'{e}nyi model}, considers a random graph $\mathcal{G}_{\textrm{ER}}\left(N,p\left(N\right)\right)$ on $N$ nodes formed by including each link between two different nodes according to independent Bernoulli trials that succeed with probability $p\left(N\right)$, which is often allowed to vary with $N$ \cite{BBol1}.  This distribution has important properties regarding asymptotic connectedness behavior as $N$ increases that condition on the relationship between $N$ and $p\left(N\right)$ \cite{EGil1}.  Additionally, the relationship between Erd\"{o}s-R\'{e}nyi random graphs and Wigner matrices leads to asymptotic spectral statistics that may be characterized \cite{XDin1}.

Similarly, starting from an arbitrary supergraph $\mathcal{G}_{\mathrm{sup},N}$ with $N$ nodes, the graph-valued random variable $\mathcal{G}_{\mathrm{perc}}\left(\mathcal{G}_{\mathrm{sup},N},p\left(N\right)\right)$ that results from including each link of $G_{\mathrm{sup},N}$ according to independent Bernoulli trials that succeed with probability $p\left(N\right)$ is known as a \definition{bond percolation model}.  Thus, an Erd\"{o}s-R\'{e}nyi model is a bond percolation model with a complete supergraph.  Typically, study of bond percolation models concerns their asymptotic connectivity behavior, although that is not the purpose of this paper.  Here, the terminology is simply used to refer to the type of random graph models considered in this paper.

\subsection{Spectral Statistics}

The eigenvalues of matrices, such as the graph adjacency matrix, that respect the network structure can be important in the design of distributed algorithms.  Hence, an understanding of the eigenvalues of relevant random symmetric matrices is desired.  For certain Hermitian random matrix models parameterized by matrix dimension $N\times N$, results are available concerning the asymptotic behavior of the eigenvalues as $N$ increases through a function called the empirical spectral distribution.  The definitions employed in this paper are provided below.

Note that the eigenvalues of a Hermitian matrix must be real valued.  Given a $N\times N$  Hermitian matrix-valued random variable $W_N$, order the $N$ random eigenvalues such that $\lambda_1\left(W\right) \leq \ldots \leq \lambda_i\left(W_N\right) \leq \lambda_{i+1}\left(W_N\right) \leq \ldots \leq \lambda_N\left(W_N\right)$.  The \definition{empirical spectral measure} of $W_N$ given by
\begin{equation}\label{SpectralStatistics:EQ:ESM}
	\mu_{w_N}\left(X\right)=\frac{1}{N}\sum_{i=1}^N \chi\left(\lambda_i\left(W_N\right)\in X\right)
\end{equation}
assigns to each subset $X\subseteq\mathbb{R}$ the fraction of eigenvalues that appear in $X$ using the indicator function $\chi$.  Related to this function, the \definition{empirical spectral distribution} of $W_N$ given by \cite{RCou1}
\begingroup
%\thinmuskip=\muexpr\thinmuskip*1/16\relax
%\medmuskip=\muexpr\medmuskip*1/16\relax
 \begin{equation}\label{SpectralStatistics:EQ:ESDi}
 \begin{aligned}
 	F_{W_N}(x)
 	&=\mu_{W_N}\left(\left(-\infty,x\right]\right)\\
 	&=\frac{1}{N}\sum_{i=1}^N \chi\left(\lambda_i\left(W_N\right)\leq x\right)
 \end{aligned}
 \end{equation}
\endgroup
counts the number of eigenvalues in the interval $\left(-\infty,x\right]$, and the \definition{empirical spectral density} of $W_N$ is given by \cite{RCou1}
 \begin{equation}\label{SpectralStatistics:EQ:ESDe}
 \begin{aligned}
	f_{W}\left(x\right)
	&=\frac{d}{dx}F_{W_N}\left(x\right)\\
	&=\frac{1}{N}\sum_{i=1}^N \delta\left(x-\lambda_i\left(W_N\right)\right).
 \end{aligned}
 \end{equation}
 indicates the locations of the eigenvalues using the Dirac delta function $\delta$.  The empirical spectral measure and empirical spectral distribution define each other and are often referred to interchangeably.  Note that each of these definitions produces a function-valued random variable.  From them can be defined the \definition{expected empirical spectral measure} $\mu_{\mathrm{exp},W_N}=\operatorname{E}\left[\mu_{W_N}\right]$, the \definition{expected empirical spectral distribution} $F_{\mathrm{exp},W_N}=\operatorname{E}\left[F_{W_N}\right]$, and the \definition{expected empirical spectral density} $f_{\mathrm{exp},W_N}=\operatorname{E}\left[f_{W_N}\right]$.

Analysis of the empirical spectral distribution for Hermitian matrices often involves the \definition{Stieltjes transform} defined as below \cite{RCou1}.
\begin{equation}
S_F\left(z\right)=\int_{-\infty}^{\phantom{-}\infty}\frac{1}{x-z}dF\left(x\right), \quad \operatorname{Im}\left\{z\right\}\neq 0
\end{equation}
The Stieltjes transform can be inverted to obtain the corresponding density by computing the following expression \cite{RCou1}.
\begin{equation}
F\left(x\right)=\lim_{\epsilon\rightarrow 0^+}\frac{1}{\pi}\int_{-\infty}^{x}\operatorname{Im}\left\{S_F\left(\lambda+\epsilon i\right)\right\}d\lambda
\end{equation}
\begin{equation}
f\left(x\right)=\lim_{\epsilon\rightarrow 0^+}\frac{1}{\pi}\operatorname{Im}\left\{S_F\left(x+\epsilon i\right)\right\}
\end{equation}
For the empirical spectral distribution $F_N$ of an $N\times N$ Hermitian matrix $W_N$, the Stieltjes function computes to
\begingroup
\thinmuskip=\muexpr\thinmuskip*1/16\relax
\medmuskip=\muexpr\medmuskip*1/16\relax
\begin{equation}
S_{F_{W_N}}\left(z\right)=\frac{1}{N}\operatorname{tr}\left(\left(W_N-zI_N\right)^{-1}\right), \enskip \operatorname{Im}\left\{z\right\}\neq 0
\end{equation}
\endgroup
which is the normalized trace of the \definition{resolvent} $\left(W_N-zI_N\right)^{-1}$ \cite{RCou1}.

Given a sequence of Hermitian matrix-valued random variables $\{W_N\}_{N=1}^{\infty}$, the sequence has \definition{limiting spectral measure} $\mu_{\mathrm{lim},W_N}$ provided $\mu_{W_N}$ converges weakly to $\mu_{\mathrm{lim},W_N}$.  Similarly, it has \definition{limiting spectral distribution} $F_{\mathrm{lim},W_N}$ provided $F_{W_N}$ converges weakly to $F_{\mathrm{lim},W_N}$.  In some circumstances in which these limits may or may not exist, the spectral statistics can still be describable by a sequence of deterministic functions.  Given a sequence of Hermitian matrix-valued random variables $W_N$, a \definition{deterministic equivalent} for $W_N$ with respect to the sequence of functionals $g_N$ is a sequence of deterministic matrices $W_{N}^{\circ}$ such that
\begin{equation}
\lim_{N\rightarrow\infty}{\left(g_N\left(W_N\right)-g_N\left(W_{N}^{\circ}\right)\right)}=0
\end{equation}
 almost surely \cite{RCou1}.  Additionally, the sequence of deterministic values $g_N\left(W_{N}^{\circ}\right)$ is called a deterministic equivalent of the sequence of random values $g_N\left(W_N\right)$~\cite{RCou1}.

%In order to study the empirical spectral distribution, the functional may be chosen to be $g_N\left(W_N\right)=S_{F_{W_N}}(z)$, as in the work of Girko described in the subsequent section.

%[wigner matrices]

\subsection{Stochastic Canonical Equations}

The most critical theorem for this work provides a method for finding deterministic equivalents of empirical spectral distribution functions for symmetric random matrices with independent entries, except for the relation that the lower triangular entries are determined by symmetry from the upper triangular entries.  Provided some regularity conditions \eqref{GirkoCond1}, \eqref{GirkoCond2}, and \eqref{GirkoCond3} hold, the Stieltjes transform \eqref{STEq} of a deterministic equivalent distribution function can be found by solving a system of equations \eqref{GirkoK01Eq} containing matrices \cite{VGir1}.  Theorem \ref{GirkoK01Thm} provides the formal statement below.  This result and a compilation of many other related results can be found in \cite{VGir1}.

\begin{theorem}[Girko's K1 Equation \cite{VGir1}]\label{GirkoK01Thm}

Consider a family of symmetric matrix valued random variables $W_N$ indexed by size $N$ such that $W_N$ is an $N\times N$ symmetric matrix in which the entries on the upper triangular region are independent.  That is, $\left\{\left(W_{N}\right)_{ij}|1\leq i\leq j\leq N\right\}$ are independent with $\left(W_{N}\right)_{ji}=\left(W_{N}\right)_{ij}$.  Let $W_N$ have expectation $B_N=\operatorname{E}\left[W_N\right]$ and centralization $H_N=W_N-E\left[W_N\right]$ such that the following three conditions hold.  Note that in order to avoid cumbersome indexing, the index $N$ will henceforth be omitted from most expressions involving $W_N$, $B_N$, and $H_N$.
\begin{gather}
\sup_{N}{\max_{i}{\sum_{j=1}^{N}{\left|B_{ij}\right|}}}<\infty \label{GirkoCond1} \\
\sup_{N}{\max_{i}{\sum_{j=1}^{N}{\operatorname{E}\left[H_{ij}^2\right]}}}<\infty \label{GirkoCond2}\\
\begin{aligned}
\lim_{N\rightarrow\infty}{\max_{i}{\sum_{j=1}^{N}{\operatorname{E}\left[H_{ij}^2\chi\left(\left|H_{ij}\right|>\tau\right)\right]}}}\\
=0\textrm{ for all }\tau>0
\end{aligned}\label{GirkoCond3}
\end{gather}
Then for almost all $x$,
\begin{equation}
\lim_{N\rightarrow\infty} \left|F_{W_N}\left(x\right)-F_N\left(x\right)\right|=0
\end{equation}
almost surely, where $F_N$ is the distribution with Stieltjes transform
\begin{equation}
\begin{aligned}\label{STEq}
S_{F_N}(z)
&=\int{\frac{1}{x-z}dF_{N}(x)}\\
&=\frac{1}{N}\sum_{k=1}^{N}C_{kk}(z),\quad \operatorname{Im}\left\{z\right\}\neq 0
\end{aligned}
\end{equation}
and the functions $C_{kk}\left(z\right)$ satisfy the canonical system of equations
\begingroup
\thinmuskip=\muexpr\thinmuskip*1/16\relax
\medmuskip=\muexpr\medmuskip*1/16\relax
\begin{equation}\label{GirkoK01Eq}
\begin{aligned}
	%C_{kk}\left(z\right)=&\left[\left(B-zI-\left(\delta_{lj}\sum_{s=1}^{N}{C_{ss}(z)\operatorname{E}\left[H_{js}^2\right]}\right)_{l,j=1}^{l,j=N}\right)^{-1}\right]_{kk}
	C_{kk}\left(z\right)=&\left[\left(B-zI-\vphantom{\left(\delta_{lj}\sum_{s=1}^{N}{C_{ss}(z)\operatorname{E}\left[H_{js}^2\right]}\right)_{l,j=1}^{l,j=N}}\right.\right.\ldots \\
	\ldots&\left.\left.\left(\delta_{lj}\sum_{s=1}^{N}{C_{ss}(z)\operatorname{E}\left[H_{js}^2\right]}\right)_{l,j=1}^{l,j=N}\right)^{-1}\right]_{kk}
\end{aligned}
\end{equation}
\endgroup
for $k=1,\ldots,N$.  Note that the notation $\left(\cdot\right)_{l,j=1}^{l,j=N}$ indicates a matrix built from the parameterized contents of the parentheses, such that $X=\left(X_{ij}\right)_{l,j=1}^{l,j=N}$, and $\delta_{lj}$ is the Kronecker delta function.  There exists a unique solution $C_{kk}(z)$ for $k=1,\ldots,N$ to the canonical system of equations \eqref{GirkoK01Eq} among the class $L=\left\{X(z)\in \mathbb{C} \mid X\left(z\right)\textrm{ analytic},~\operatorname{Im}\left\{z\right\}\operatorname{Im}\left\{X\left(z\right)\right\}>0\right\}$.  Furthermore, if
\begin{equation}
\inf_{i,j} {N\operatorname{E}\left[H_{ij}^2\right]}\geq c >0, \label{GirkoCond4}
\end{equation}
then
\begin{equation}
\lim_{N\rightarrow\infty} {\sup_{x}{\left|F_{W_N}\left(x\right)-F_N\left(x\right)\right|}}=0
\end{equation}
almost surely, where $F_N$ is defined as above.

\end{theorem}

In addition to the independence properties of the entries of the random matrix, Theorem \ref{GirkoK01Thm} specifies three conditions that must be verified.  The first condition \eqref{GirkoCond1} bounds the absolute sum of the entries of each row, and the second condition \eqref{GirkoCond2} bounds the total variance of entries of each row.  The third condition \eqref{GirkoCond3} is closely related to Lindberg's condition of the central limit theorem.  Under these conditions, a deterministic equivalent $F_N(x)$ for $F_{W_N}(x)$ can be found by solving the system of equations \eqref{GirkoK01Eq} and computing the Stieltjes transform via \eqref{STEq}, which can then be inverted.  Under the additional condition \eqref{GirkoCond4}, the supremum converges~\cite{VGir1}.

\section{Main Results}\label{SecMainResults}

This section examines, in particular, a random graph model formed by percolation of a $D$-dimensional lattice graph.  Note that there are many types of graphs known as lattice graphs in other contexts, so the precise definition used will be introduced here.  A \definition{$\mathbf{D}$-dimensional lattice graph} with size $M_d$ along the $d$th dimension is a graph $\mathcal{G}_{\mathrm{lat}}=\left(\mathcal{V},\mathcal{E}\right)$ in which the $\left|\mathcal{V}\right|=N=\prod_{d=1}^{D}{M_d}$ nodes are identified with the ordered $D$-tuples in $\mathbb{Z}\left[1,M_1\right]\times\ldots\times\mathbb{Z}\left[1,M_D\right]$ and in which two nodes are connected by a link if the corresponding $D$-tuples differ by exactly one symbol \cite{RLas1}.  In this way, if the node indices are fixed along $D-1$ of the lattice dimensions, the node-induced subgraph along the remaining dimension is a complete graph.  Denoting the adjacency matrix of a complete graph on $M$ vertices by $K_M$, in terms of Kronecker tensor products and adjacency matrices, these lattice graphs can be described by
\begingroup
\thinmuskip=\muexpr\thinmuskip*1/16\relax
\medmuskip=\muexpr\medmuskip*1/16\relax
\begin{equation}
A\left(\mathcal{G}_{\mathrm{lat}}\right)=\sum_{j=1}^{D}{\bigotimes_{d=1}^{D}{X_{dj}}}, \enskip X_{dj}=\left\{\begin{array}{ll} K_{M_d} & j=d \\ {\phantom{K}\mllap{I}}_{M_d} & j\neq d  \end{array}\right.
\end{equation}
\endgroup
where the $j$th term in the summation contributes all complete graphs along the $j$th lattice dimension.  Such graphs are of interest because they can be grown to a large number of nodes by increasing the lattice dimension sizes while the number of distinct eigenvalues of the adjacency matrix does not increase.  The eigenvalues of the lattice graph adjacency matrix are
\begin{equation}
\begin{aligned}
\lambda\left(j_1,\ldots,j_D\right)&=\sum_{d=1}^D\lambda_d\left(j_d\right), \\ \lambda_d(j_d)&=\left\{\begin{array}{ll} M_d-1 & j_d=0 \\ \phantom{M_d-1}\mathllap{-1} & j_d=1\end{array}\right.
\end{aligned}
\end{equation}
for $j_1,\ldots,j_D=0,1$.  Thus, the number of distinct eigenvalues of a $D$-dimensional lattice graph depends only on the number of dimensions $D$ and is at most $2^D$.

The adjacency matrices of random graphs distributed according to percolation models $\mathcal{G}_{\mathrm{perc}}\left(\mathcal{G}_{\mathrm{lat}},p\right)$ of $D$-dimensional lattices are symmetric matrices with independent entries, except for relations determined by symmetry.  In more precise terms, the matrix entries $A_{ij}(\mathcal{G}_{\mathrm{perc}})$ for $i\leq j$ are independent, and $A_{ji}(\mathcal{G}_{\mathrm{perc}})=A_{ij}(\mathcal{G}_{\mathrm{perc}})$.  Thus, the scaled adjacency matrix eigenvalues of percolations models formed from these lattices can be analyzed by the tools provided in Theorem \ref{GirkoK01Thm}.  For a lattice with arbitrary number of dimensions $D$ and sizes $M_d$ for $d=1,\ldots,D$ along with an arbitrary link inclusion probability $p$, Theorem \ref{MainThm} derives the form of the unique solution $C_{kk}\left(z\right)$ for $k=1,\ldots, N$ to the system of equations \eqref{GirkoK01Eq}.  An outline of the methods used can be found in the initial paragraph of the proof.  This result will subsequently be used in Corollary \ref{MainCor} to compute the Stieltjes transform of a deterministically equivalent distribution function for the empirical spectral distribution.

\begin{theorem}[Solution Form for $D$-Lattice Percolation]\label{MainThm}

Consider the $D$-dimensional lattice graph $\mathcal{G}_{\mathrm{lat}}$ with $N$ nodes in which the $d$th dimension of the lattice had size $M_d$ for $d=1,\ldots, D$ such that the adjacency matrix is
\begingroup
\thinmuskip=\muexpr\thinmuskip*1/8\relax
\medmuskip=\muexpr\medmuskip*1/8\relax
\begin{gather}
A\left(\mathcal{G}_{\mathrm{lat}}\right)=\sum_{j=1}^{D}{\bigotimes_{d=1}^{D}{X_{dj}}}, \enskip X_{dj}=\left\{\begin{array}{ll} K_{M_d} & j=d \\ {\phantom{K}\mllap{I}}_{M_d} & j\neq d  \end{array}\right. \\ N=\prod_{d=1}^{D}{M_d}.
\end{gather}
\endgroup
Form a random graph $\mathcal{G}_{\mathrm{perc}}\left(\mathcal{G}_{\mathrm{lat}},p\right)$ by independently including each link of $\mathcal{G}_{\mathrm{lat}}$ with probability $p$, and denote the corresponding random scaled adjacency matrix $W$, expectation $B$, and centralization $H$
\begin{gather}
W\left(\mathcal{G}_{\mathrm{perc}}\right)=\frac{1}{\gamma}A\left(\mathcal{G}_{\mathrm{perc}}\right)\\
B=\operatorname{E}\left[W\left(\mathcal{G}_{\mathrm{perc}}\right)\right]\\
H\left(\mathcal{G}_{\mathrm{perc}}\right)=W\left(\mathcal{G}_{\mathrm{perc}}\right)-\operatorname{E}\left[W\left(\mathcal{G}_{\mathrm{perc}}\right)\right]
\end{gather}
where
\begin{equation}
\gamma=\gamma\left(M_1,\ldots,M_D\right)=\sum_{d=1}^{D}\left(M_d-1\right).
\end{equation}
Let $C_{kk}\left(z\right)$ for $k=1,\ldots,N$ be the unique solution to the system of equations \eqref{GirkoK01Eq} among the class $L$ guaranteed to exist by Theorem \ref{GirkoK01Thm}, and write
\begin{equation}\label{GirkoK01Sys2}
\begin{aligned}
%C\left(z\right)=\left(B-zI_N-\left(\delta_{lj}\sum_{s=1}^{N}{C_{ss}\left(z\right)\operatorname{E}\left[H_{js}^{2}\right]}\right)_{l,j=1}^{l,j=N}\right)^{-1}.
C\left(z\right)=&\left(B-zI_N-\ldots\vphantom{\left(\delta_{lj}\sum_{s=1}^{N}{C_{ss}\left(z\right)\operatorname{E}\left[H_{js}^{2}\right]}\right)_{l,j=1}^{l,j=N}}\right. \\
\ldots&\left.\left(\delta_{lj}\sum_{s=1}^{N}{C_{ss}\left(z\right)\operatorname{E}\left[H_{js}^{2}\right]}\right)_{l,j=1}^{l,j=N}\right)^{-1}.
\end{aligned}
\end{equation}
Note that $C_{kk}\left(z\right)$ is the $k$th diagonal entry of $C\left(z\right)$ and that uniqueness of $C_{kk}\left(z\right)$ implies uniqueness of $C\left(z\right)$.  For some values of $\alpha_{i_1,\ldots,i_D}\left(z\right)$ for $i_1,\ldots,i_D=0,1$
\begin{equation}
\begin{gathered}
C\left(z\right)=\sum_{\mathclap{i_1,\ldots,i_D=0}}^{1}{ \alpha_{i_1,\ldots,i_D}\left(z\right)\bigotimes_{d=1}^{D}{Y_{di_d}}}, \\ Y_{di_d}=\left\{\begin{array}{ll} {\phantom{K}\mllap{K}}_{M_d} & i_d=0 \\{\phantom{K}\mllap{I}}_{M_d} & i_d=1  \end{array}\right..
\end{gathered}\label{SolutionForm}
\end{equation}

\end{theorem}

\begin{proof}

The theorem is proven primarily by use of symmetry in the random graph model and is organized in three parts below.  The first paragraph verifies that the three conditions of Theorem \ref{GirkoK01Thm}  hold for the random matrix model $W\left(\mathcal{G}_{\mathrm{perc}}\right)$, guaranteeing existence of a unique solution $C_{kk}\left(z\right)$ for $k=1,\ldots N$ to \eqref{GirkoK01Eq}.  Subsequently, the second paragraph demonstrates that $C\left(z\right)$ is of the desired form if and only if it is invariant to permutations that factor into Kronecker products of smaller permutations along each lattice dimension.  Finally, the third paragraph uses symmetries in the random graph model to show that $C\left(z\right)$, is, in fact, invariant under these  permutations, yielding the result in the theorem.

In order to simplify the calculations used to verify the conditions of Theorem \ref{GirkoK01Thm}, a more explicit characterization of the entries $A_{ij}$ is first introduced.  Note that every integer $1\leq x \leq N$ can be uniquely represented in a mixed-radix system as
\begin{equation}
x=1+\sum_{d=1}^{D}{\beta\left(x,d\right)\left(\prod_{j=1}^{d-1}{M_j}\right)}
\end{equation}
for integers $0\leq \beta\left(x,d\right)\leq M_d-1$.  Hence, node indices $(i,j)$ can be described by their digit sequences $\beta\left(i\right)=\left\{\beta\left(i,d\right)\right\}_{d=1}^{D}$ and $\beta\left(j\right)=\left\{\beta\left(j,d\right)\right\}_{d=1}^{D}$ in this system.  Let $\left\|\beta\left(i\right)-\beta\left(j\right)\right\|_0$ count the differences between these digit sequences.  Using this to describe the entries of the supergraph adjacency matrix yields
\begin{equation}
A_{ij}\left(\mathcal{G}_{\mathrm{lat}}\right)=\left\{\begin{array}{cc} 1 & \left\|\beta\left(i\right)-\beta\left(j\right)\right\|_0=1 \\ 0 & \textrm{otherwise} \end{array}\right..
\end{equation}
Because $B_{ij}=\frac{p}{\gamma}A_{ij}\left(\mathcal{G}_{\mathrm{lat}}\right)$, the following computation verifies the first condition \eqref{GirkoCond1} of Theorem \ref{GirkoK01Thm}.
\begingroup
\begin{equation}
\begin{aligned}
\sup_{N}\ \max_{i}\sum_{j=1}^{N}{\left|B_{ij}\right|}
&=\sup_{N}{\ \max_{i}{\sum_{\left\|\beta\left(i\right)-\beta\left(j\right)\right\|_0=1}{\hspace{-1em}\left|B_{ij}\right|}}}\\
&=\sup_{N}{\frac{p}{\gamma}\sum_{d=1}^{D}\left(M_d-1\right)}\\
&=p<\infty
\end{aligned}\label{Verify_Cond1}
\end{equation}
\endgroup
Noting that $\operatorname{E}\left[\right(H_{ij}\left)_{xy}^2\right]=\frac{p\left(1-p\right)}{\gamma^2}$ if $\left(A_{ij}\left(\mathcal{G}_{\mathrm{lat}}\right)\right)_{xy}=1$ and that $\operatorname{E}\left[\right(H_{ij}\left)_{xy}^2\right]=0$ if $\left(A_{ij}\left(\mathcal{G}_{\mathrm{lat}}\right)\right)_{xy}=0$, the following computation verifies the second condition \eqref{GirkoCond2} of Theorem \ref{GirkoK01Thm}.
\begin{equation}
\begin{aligned}
\sup_{N}\max_{i}\sum_{j=1}^{N}{\operatorname{E}\left[H_{ij}^2\right]}
&=\sup_{N}\max_{i}{\hspace{-1em}\sum_{\left\|\beta\left(i\right)-\beta\left(j\right)\right\|_0=1}{\hspace{-2em}\operatorname{E}\left[H_{ij}^2\right]}}\\
&\hspace{-3em}=\sup_{N}\frac{p\left(1-p\right)}{\gamma^2}\left(\sum_{d=1}^{D}\left(M_d-1\right)\right)\\
&\hspace{-3em}\leq p\left(1-p\right)<\infty
\end{aligned}\label{Verify_Cond2}
\end{equation}
Finally, the following computation demonstrates that the third condition \eqref{GirkoCond3} of Theorem \ref{GirkoK01Thm} holds for any $\tau>0$.
\begin{equation}
\begin{aligned}
&\lim_{N\rightarrow\infty}{\max_{i}{\sum_{j=1}^{N}{\operatorname{E}\left[H_{ij}^2\chi\left(\left|H_{ij}\right|>\tau\right)\right]}}}\leq \\
&\lim_{N\rightarrow\infty}{\max_{i}{\hspace{-.8em}\sum_{\left\|\beta\left(i\right)-\beta\left(j\right)\right\|_0=1}{\hspace{-1em}\operatorname{E}\left[\frac{\max\left(p^2,\left(1-p\right)^2\right)}{\gamma^2}\ldots\right.}}}\\
&\hspace{5.25em}{\left.\ldots\chi\left(\vphantom{\frac{\max\left(p^2,\left(1-p\right)^2\right)}{\gamma^2}}\frac{\max\left(p,1-p\right)}{\gamma}>\tau\right)\right]}=0
\end{aligned}\label{Verify_Cond3}
\end{equation}
This is true because
\begin{equation}
\lim_{N\rightarrow\infty}\frac{\max\left(p,1-p\right)}{\gamma}=0,
\end{equation}
so when $\gamma$ becomes sufficiently large
\begin{equation}
\lim_{N\rightarrow\infty}\chi\left(\frac{\max\left(p,1-p\right)}{\gamma}>\tau\right)=0.
\end{equation}
Note that because $H_{ij}=0$ when $\left\|\beta\left(i\right)-\beta\left(j\right)\right\|_0\neq 1$ the additional condition \eqref{GirkoCond4} for the stronger conclusion of Theorem \ref{GirkoK01Thm} does not hold.

Given that $C(z)$ is of the form in \eqref{SolutionForm}, let $\phi_d$ be a permutation function on $M_d$ symbols with corresponding row permutation matrix $P_d$ for $d=1,\ldots,D$ and $P=P_1 \otimes \ldots \otimes P_D$ with corresponding permutation function $\phi$.  It is clear that
\begingroup
\thinmuskip=\muexpr\thinmuskip*1/16\relax
\medmuskip=\muexpr\medmuskip*1/8\relax
%\begin{equation}
\begin{align}
\hspace{-.5em}P^{\vphantom{\top}}C\left(z\right)P^{\top}
&=P^{\vphantom{\top}}\left(\sum_{{i_1,\ldots,i_D=0}}^{1}{\hspace{-1em}\alpha_{i_1,\ldots,i_D}\left(z\right)\bigotimes_{d=1}^{D}{Y_{di_d}}}\right)P^{\top}\nonumber\\
&=\phantom{P^{\vphantom{\top}}+}\sum_{{i_1,\ldots,i_D=0}}^{1}{\hspace{-1em}\alpha_{i_1,\ldots,i_D}\left(z\right)\bigotimes_{d=1}^{D}{P_d^{\vphantom{\top}}Y_{di_d}P_d^{\top}}}\\
&=\phantom{P^{\vphantom{\top}}+}\sum_{{i_1,\ldots,i_D=0}}^{1}{\hspace{-1em}\alpha_{i_1,\ldots,i_D}\left(z\right)\bigotimes_{d=1}^{D}{Y_{di_d}}}=C(z).\nonumber
\end{align}
%\end{equation}
\endgroup
Note, again, that every integer $1\leq x\leq N$ can be uniquely represented as
\begin{equation}
x=1+\sum_{d=1}^D{\beta\left(x,d\right)\left(\prod_{j=1}^{d-1}M_j\right)}
\end{equation}
for integers $0\leq \beta\left(x,d\right)< M_d$ and that
\begingroup
\thinmuskip=\muexpr\thinmuskip*1/2\relax
\medmuskip=\muexpr\medmuskip*1/2\relax
\begin{equation}
\phi\left(x\right)=1+\sum_{d=1}^D{\left(\phi_d\left(\beta\left(x,d\right)+1\right)-1\right)\left(\prod_{j=1}^{d-1}M_j\right)}.
\end{equation}
\endgroup
If $C(z)$ is of the form in \eqref{SolutionForm} then
$
C(z)_{xy}=\alpha_{\delta_{\beta\left(x,1\right)\beta\left(y,1\right)},\ldots,\delta_{\beta\left(x,D\right)\beta\left(y,D\right)}}\left(z\right).
$
If $C(z)$ is not of the form in \eqref{SolutionForm} then there are some $x_1,y_1$ and $x_2,y_2$ such that $\delta_{\beta\left(x_1,d\right)\beta\left(y_1,d\right)}=\delta_{\beta\left(x_2,d\right)\beta\left(y_2,d\right)}$ for $d=1,\ldots,D$ but that $C(z)_{x_1,y_1}\neq C(z)_{x_2,y_2}$.  Let $\psi_1,\ldots,\psi_D$ be permutations such that $\psi_d\left(\beta\left(x_1,d\right)+1\right)-1=\beta\left(x_2,d\right)$ and $\psi_d\left(\beta\left(y_1,d\right)+1\right)-1=\beta\left(y_2,d\right)$ with corresponding row permutation matrices $Q_d$.  Let $Q=Q_1\otimes\ldots Q_D$ be the row permutation matrix corresponding to permutation $\psi$.  Then
\begingroup
%\thinmuskip=\muexpr\thinmuskip*1/8\relax
%\medmuskip=\muexpr\medmuskip*1/8\relax
\begin{equation}
\begin{aligned}
\left(Q^{\vphantom{\top}}C\left(z\right)Q^{\top}\right)_{x_1y_1}&=C\left(z\right)_{\psi\left(x_1\right)\psi\left(y_1\right)}\\
&=C\left(z\right)_{x_2y_2}\neq C\left(z\right)_{x_1y_1}.
\end{aligned}
\end{equation}
\endgroup
Hence, $C\left(z\right)$ is of the form in \eqref{SolutionForm} if and only if $C\left(z\right)=P^{\vphantom{\top}}C\left(z\right)P^{\top}$ for all $P=P_1\otimes\ldots\otimes P_D$ where $P_d$ is a row permutation matrix on $M_d$ symbols.

Let $\phi_d$ be a permutation function on $M_d$ symbols with corresponding row permutation matrix $P_d$ for $d=1,\ldots,D$.  Also, let $P=P_1 \otimes \ldots \otimes P_D$, thus forming a permutation matrix that operates on each lattice dimension.  The following equations result from permuting the rows and columns of $C\left(z\right)$ and by manipulating the matrix inverse, expectations, and transposes.
\begingroup
\thinmuskip=\muexpr\thinmuskip*1/8\relax
\medmuskip=\muexpr\medmuskip*1/8\relax
\begin{equation}
\begin{aligned}
P^{\vphantom{\top}}C\left(z\right)P^{\top}
&=P^{\vphantom{\top}}\left(B-zI_N-\ldots\vphantom{\phantom{P^{\vphantom{\top}}}\left(\delta_{lj}\sum_{s=1}^{N}{C_{ss}\left(z\right)\operatorname{E}\left[H_{js}^{2}\right]}\right)_{l,j=1}^{l,j=N}\phantom{P^{\top}}}\right.\\
&\hspace{-2em}\left.\ldots\left(\delta_{lj}\sum_{s=1}^{N}{C_{ss}\left(z\right)\operatorname{E}\left[H_{js}^{2}\right]}\right)_{l,j=1}^{l,j=N}\right)^{-1}\mathclap{P^{\top}}\\
&=\phantom{P^{\vphantom{\top}}}\left(P^{\vphantom{\top}}BP^{\top}-zP^{\vphantom{\top}}I_NP^{\top}-\ldots\vphantom{\left(\delta_{lj}\sum_{s=1}^{N}{C_{ss}\left(z\right)\operatorname{E}\left[H_{\phi\left(j\right)s}^{2}\right]}\right)_{l,j=1}^{l,j=N}}\right.\\
&\hspace{-2em}\ldots\left.\phantom{P^{\vphantom{\top}}}\left(\delta_{lj}\sum_{s=1}^{N}{C_{ss}\left(z\right)\operatorname{E}\left[H_{\phi\left(j\right)s}^{2}\right]}\right)_{l,j=1}^{l,j=N}\right)^{-1}
\end{aligned}
\label{PermuteC}
\end{equation}
\endgroup
Note that for the underlying lattice graph, $P^{\vphantom{\top}}A\left(\mathcal{G}_{\mathrm{lat}}\right)P^{\top}=A\left(\mathcal{G}_{\mathrm{lat}}\right)$.  Hence $W\left(\mathcal{G}_{\mathrm{perc}}\right)$ has the same distribution as $P^{\vphantom{\top}}W\left(\mathcal{G}_{\mathrm{perc}}\right)P^{\top}$ as all edge inclusions are independent and identically distributed.  Therefore,
\begin{equation}
B=P^{\vphantom{\top}}BP^{\top}.\label{Symmetry1}
\end{equation}
Furthermore, in terms of matrix entries, $W_{\phi\left(x\right)\phi\left(y\right)}\left(\mathcal{G}_{\mathrm{perc}}\right)$ has the same distribution as $W_{xy}\left(\mathcal{G}_{\mathrm{perc}}\right)$.  Hence, $H_{\phi\left(j\right)s}$ has the same distribution as $H_{j\phi^{-1}\left(s\right)}$.  By applying these two identities and changing the order of summation by $\phi$, the following equations result.
\begingroup
\thinmuskip=\muexpr\thinmuskip*1/8\relax
\medmuskip=\muexpr\medmuskip*1/8\relax
\begin{equation}
\begin{aligned}
P^{\vphantom{\top}}C\left(z\right)P^{\top}&=\left(B-zI_N-\ldots\vphantom{\left(\delta_{lj}\sum_{s=1}^{N}{C_{ss}\left(z\right)\operatorname{E}\left[H_{j\phi^{-1}\left(s\right)}^{2}\right]}\right)_{l,j=1}^{l,j=N}}\right.\\
&\hspace{-2em}\ldots\left.\left(\delta_{lj}\sum_{s=1}^{N}{C_{ss}\left(z\right)\operatorname{E}\left[H_{j\phi^{-1}\left(s\right)}^{2}\right]}\right)_{l,j=1}^{l,j=N}\right)^{-1}\\
&=\left(B-zI_N-\ldots\vphantom{\left(\delta_{lj}\sum_{s=1}^{N}{C_{\phi\left(s\right)\phi\left(s\right)}\left(z\right)\operatorname{E}\left[H_{js}^{2}\right]}\right)_{l,j=1}^{l,j=N}}\right.\\
&\hspace{-2em}\ldots\left.\left(\delta_{lj}\sum_{s=1}^{N}{C_{\phi\left(s\right)\phi\left(s\right)}\left(z\right)\operatorname{E}\left[H_{js}^{2}\right]}\right)_{l,j=1}^{l,j=N}\right)^{-1}
\end{aligned}
\end{equation}
\endgroup
The entries $C_{\phi\left(k\right)\phi\left(k\right)}\left(z\right)$ are the $k$th diagonal entries of $P^{\vphantom{\top}}C\left(z\right)P^{\top}$ for $k=1,\ldots,N$ and, thus, a solution to the system of equations \eqref{GirkoK01Eq}.  However, because the solution is unique, this implies that
\begin{equation}
C\left(z\right)=P^{\vphantom{\top}}C\left(z\right)P^{\top}.
\end{equation}
Therefore, $C\left(z\right)$ is of the form shown in \eqref{SolutionForm}.\hfill $\blacksquare$

\end{proof}

While Theorm \ref{MainThm} specifies the form of the unique solution to the system of equations \eqref{GirkoK01Eq}, it does not describe how to find the $2^D$ parameters $\alpha_{i_1,\ldots,i_D}\left(z\right)$ that correspond to a given set of lattice dimensions and link probability parameter that are necessary to find the Stieltjes transform of the deterministically equivalent distribution function $F_N$.  Corollary \ref{MainCor} shows that the Stieltjes transform of $F_N$ is equal to the diagonal element of $C(z)$ described in Theorem \ref{MainThm}.  It also derives the system of $2^D$ nonlinear equations to which the parameters are a solution, accomplishing this by computing the terms of \eqref{GirkoK01Sys2} and noting that all computed matrices are simultaneously diagonalizable.  The matrix equation in \eqref{GirkoK01Sys2} then transforms into a system of equations that describes the eigenvalues, which appears in \eqref{CoeffSystem} of Corollary \ref{MainCor}.  Remark \ref{IterRmk} describes use of an iterative approach to find the solution.  For added clarity, please note that $i_1,\ldots,i_D=0,1$ index the parameters $\alpha_{i_1,\ldots,i_D}\left(z\right)$ while $j_1,\ldots,j_D=0,1$ index the equations that describe the parameters.

\begin{corollary}[Stieltjes Transform for $D$-Lattice Percolation]\label{MainCor}

The Stieltjes transform of the deterministic equivalent distribution function $F_n$ specified in Theorem \ref{GirkoK01Thm} for the empirical spectral distribution of $W\left(\mathcal{G}_{\mathrm{perc}}\right)$ is given by
\begin{equation}
S_{F_N}(z)=\alpha_{1,\ldots,1}\left(z\right), \quad \operatorname{Im}\left(z\right)\neq 0
\end{equation}
where the $2^D$ complex variables $\alpha_{i_1,\ldots,i_D}(z)$ for $i_1,\ldots,i_D=0,1$ solves the system
\begingroup
%\small
\thinmuskip=\muexpr\thinmuskip*1/32\relax
\medmuskip=\muexpr\medmuskip*1/32\relax
\begin{equation}
\begin{aligned}
%\pushleft{\hspace{1em}\sum_{\mathclap{i_1,\ldots,i_D=0}}^1{\alpha_{i_1,\ldots,i_D}(z)\prod_{d=1}^{D}{\lambda_{di_d}\left(j_d\right)}}=}\\
%\pushleft{\frac{1}{\frac{p}{\gamma}\left(\sum\limits_{d=1}^{D}\lambda_{d0}(j_d)\right)-z-\frac{p\left(1-p\right)}{\gamma^2}\left(\sum\limits_{d=1}^{D}{\left(M_d-1\right)}\right)\alpha_{1,\ldots,1}\left(z\right)}}
\hspace{1em}\sum_{\mathclap{i_1,\ldots,i_D=0}}^1\alpha_{i_1,\ldots,i_D}&(z)\prod_{d=1}^{D}{\lambda_{di_d}\left(j_d\right)}=\\
&\hspace{-3em}\left(\frac{p}{\gamma}\left(\sum\limits_{d=1}^{D}\lambda_{d0}(j_d)\right)-z-\ldots\right.\\
&\hspace{-3em}\left.\ldots\frac{p\left(1-p\right)}{\gamma^2}\left(\sum\limits_{d=1}^{D}{\left(M_d-1\right)}\right)\alpha_{1,\ldots,1}\left(z\right)\right)^{-1}
\end{aligned}
\label{CoeffSystem}
\end{equation}
\endgroup
of $2^D$ rational equations for $j_1,\ldots,j_D=0,1$ where
\begin{equation}
\lambda_{di_d}(j_d)=\left\{\begin{array}{ll} M_d-1 &  i_d=0,j_d=0 \\ \phantom{M_d-1}\mllap{-1} & i_d=0,j_d=1 \\ \phantom{M_d-1}\mllap{1} & i_d=1\phantom{j_d=0}\end{array}\right..
\end{equation}

\end{corollary}

\begin{proof}

Using the Stieltjes transform expression \eqref{STEq} from Theorem \ref{GirkoK01Thm} and the derived form for $C(z)$ in \eqref{SolutionForm} from Theorem \ref{MainThm}, the Stieltjes transform of the deterministic equivalent distribution function $F_N$ specified in Theorem \ref{GirkoK01Thm} for the empirical spectral distribution of $W\left(\mathcal{G}_{\mathrm{perc}}\right)$ is computed as
\begingroup
\begin{equation}
\begin{aligned}
S_{F_N}\left(z\right)
&=\frac{1}{N}\sum_{k=1}^{N}{C_{kk}\left(z\right)}\\
&=\frac{1}{N}\sum_{k=1}^{N}{\alpha_{1,\ldots,1}\left(z\right)}
=\alpha_{1,\ldots,1}\left(z\right).
\end{aligned}
\end{equation}
\endgroup

In order to derive the system of equations describing the coefficients $\alpha_{i_1,\ldots,i_D}$ for $i_1,\ldots,i_D=0,1$, the value of the third matrix in the denominator of \eqref{GirkoK01Sys2} must first be computed.  Subsequently, the system of equations may be found by noting that all matrices in equation \eqref{GirkoK01Eq} may be simultaneously diagonalized by orthogonal matrices of eigenvectors.  Each distinct eigenvalue of $C\left(z\right)$ will result in an equation that the coefficients must satisfy.  Let $\beta$ be defined as in the proof of Theorem \ref{MainThm}.  Summing the variance over each row yields
\begingroup
\thinmuskip=\muexpr\thinmuskip*1/16\relax
\medmuskip=\muexpr\medmuskip*1/16\relax
\begin{equation}
\begin{aligned}
\sum_{s=1}^{N}{C_{ss}\left(z\right)\operatorname{E}\left[H_{js}^{2}\right]}
&=\hspace{-1em}\sum_{\left\|\beta\left(s\right)-\beta\left(j\right)\right\|_0=1}^{N}{\hspace{-1em}C_{ss}\left(z\right)\operatorname{E}\left[H_{js}^{2}\right]}\\
&\hspace{-3em}=\frac{p\left(1-p\right)}{\gamma^2}\left(\sum_{d=1}^{D}{\left(M_d-1\right)}\right)\alpha_{1,\ldots,1}\left(z\right)
\end{aligned}
\end{equation}
\endgroup
and, therefore, that
\begin{equation}
\begin{aligned}
\pushleft{\left(\delta_{lj}\sum_{s=1}^{N}{C_{ss}\left(z\right)\operatorname{E}\left[H_{js}^{2}\right]}\right)_{l,j=1}^{l,j=N}=}\\
\pushright{\hspace{2em}\frac{p\left(1-p\right)}{\gamma^2}\left(\sum_{d=1}^{D}{\left(M_d-1\right)}\right)\alpha_{1,\ldots,1}\left(z\right)I_N.}\label{DMat}
\end{aligned}
\end{equation}
Note that $C(z)$, $B$, and the matrix computed in \eqref{DMat} are simultaneously diagonalizable by Kronecker product properties.  Let $\lambda_{di_d}(0)$ be the eigenvalue of $Y_{di_d}$ corresponding to eigenvector $\mathbf{v}=\mathbf{1}$.  Let $\lambda_{di_d}(1)$ be the eigenvalue of $Y_{di_d}$ corresponding to any other eigenvector orthogonal to $\mathbf{v}=\mathbf{1}$.  Hence,
\begin{equation}
\lambda_{di_d}(j_d)=\left\{\begin{array}{ll} M_d-1 &  i_d=0,j_d=0 \\ \phantom{M_d-1}\mllap{-1} & i_d=0,j_d=1 \\ \phantom{M_d-1}\mllap{1} & i_d=1\phantom{j_d=0}\end{array}\right..
\end{equation}
The eigenvalues of $C\left(z\right)$ are indexed by $j_1,\ldots, j_D=0,1$ and are given by
\begin{equation}
\sum_{\mathclap{i_1,\ldots,i_D=0}}^1{\alpha_{i_1,\ldots,i_D}(z)\prod_{d=1}^{D}{\lambda_{di_d}\left(j_d\right)}}
\end{equation}
The corresponding eigenvalues of $B$ are given by
\begin{equation}
\frac{p}{\gamma}\left(\sum\limits_{d=1}^{D}\lambda_{d0}(j_d)\right)
\end{equation}
The corresponding eigenvalues of the matrix in \eqref{DMat} are given by
\begin{equation}
\frac{p\left(1-p\right)}{\gamma^2}\left(\sum_{d=1}^{D}{\left(M_d-1\right)}\right)\alpha_{1,\ldots,1}\left(z\right).
\end{equation}
Of course, the corresponding eigenvalue of $zI$ is $z$.  Through applying this diagonalization to the equation \eqref{GirkoK01Eq}, the system of $2^D$ equations indexed by $j_1,\ldots,j_D=0,1$ that the $2^D$ coefficients $\alpha_{i_1,\ldots,i_D}$ indexed by $i_1,\ldots,i_D=0,1$ satisfy is described by equation \eqref{CoeffSystem}.
\hfill $\blacksquare$

\end{proof}

\begin{figure*}[t]
\centering
\begin{tabular}{cccc}
\multicolumn{2}{c}{\includegraphics[width=.4\textwidth]{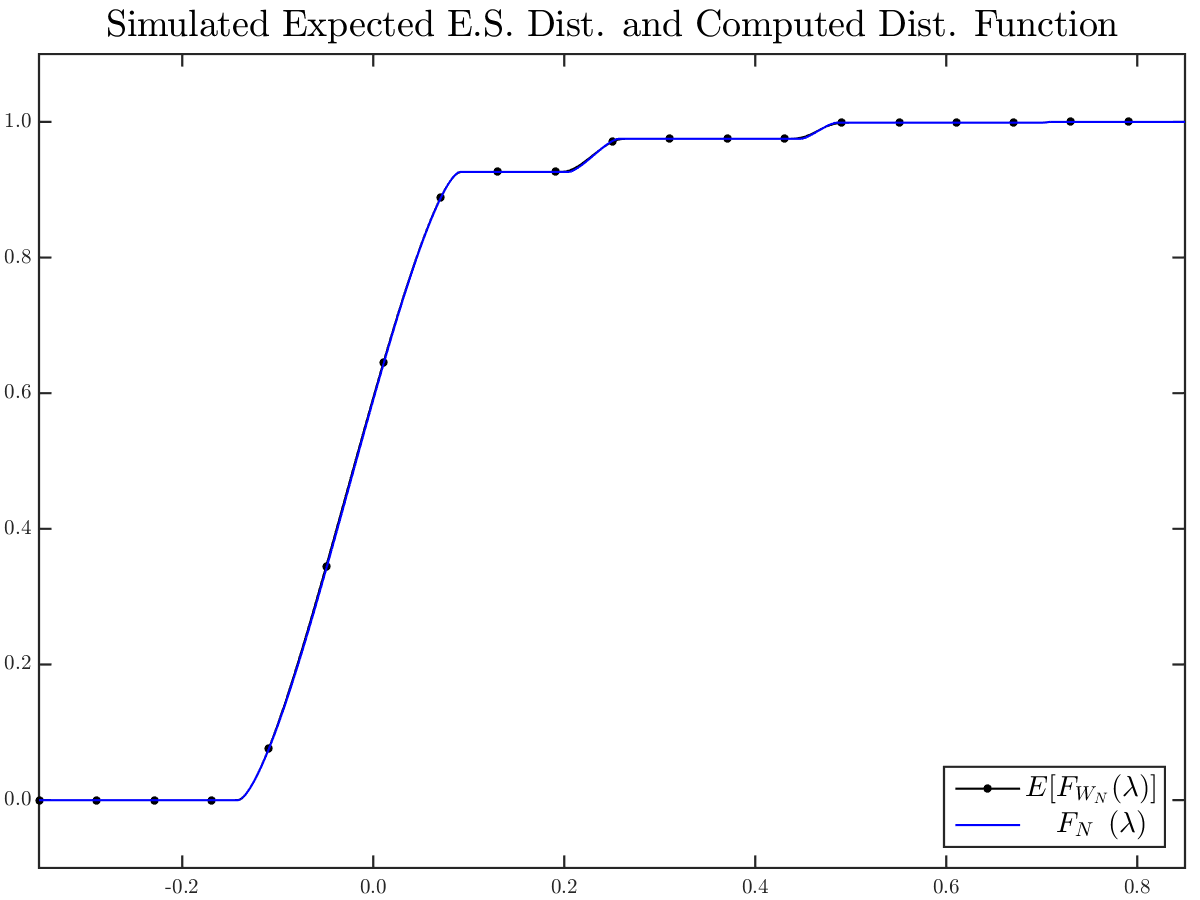}} & \multicolumn{2}{c}{\includegraphics[width=.4\textwidth]{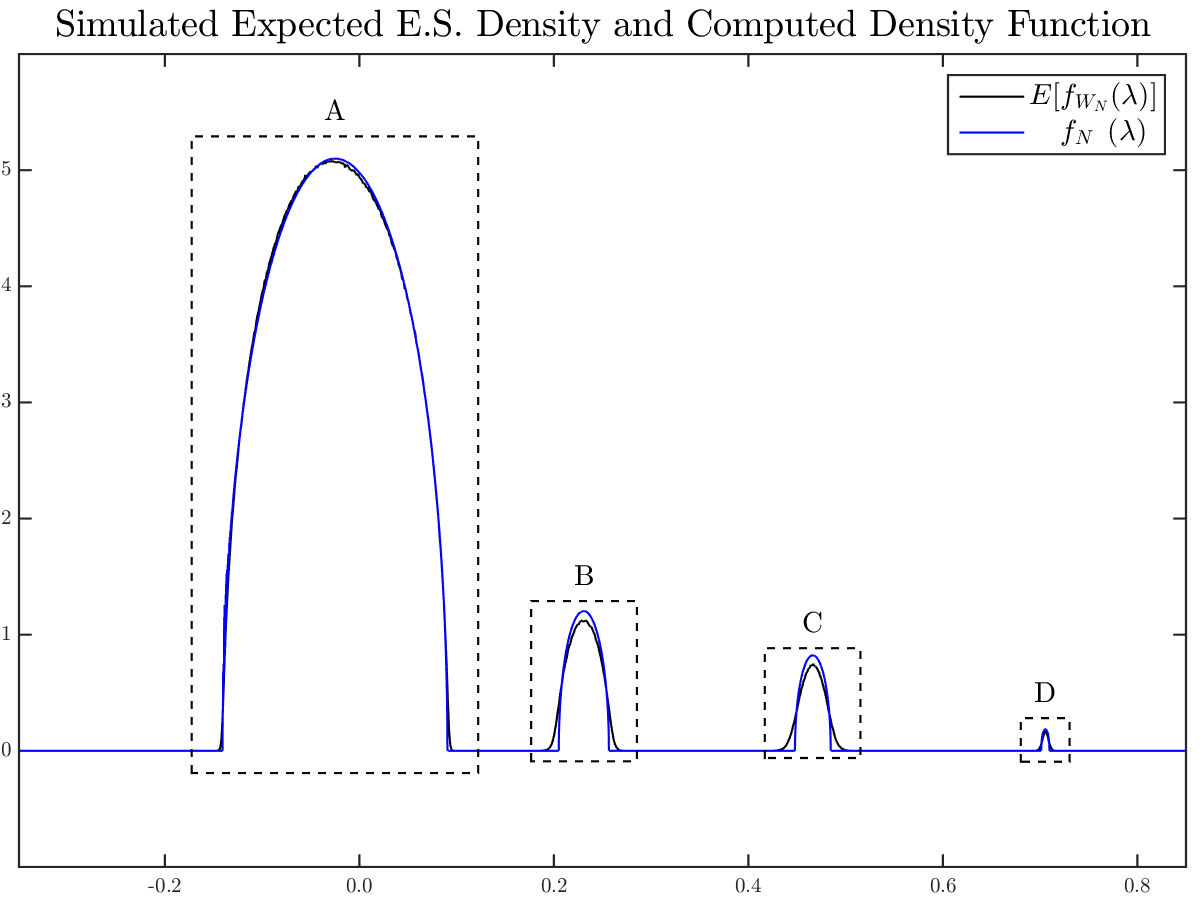}}\\
\includegraphics[width=.2\textwidth]{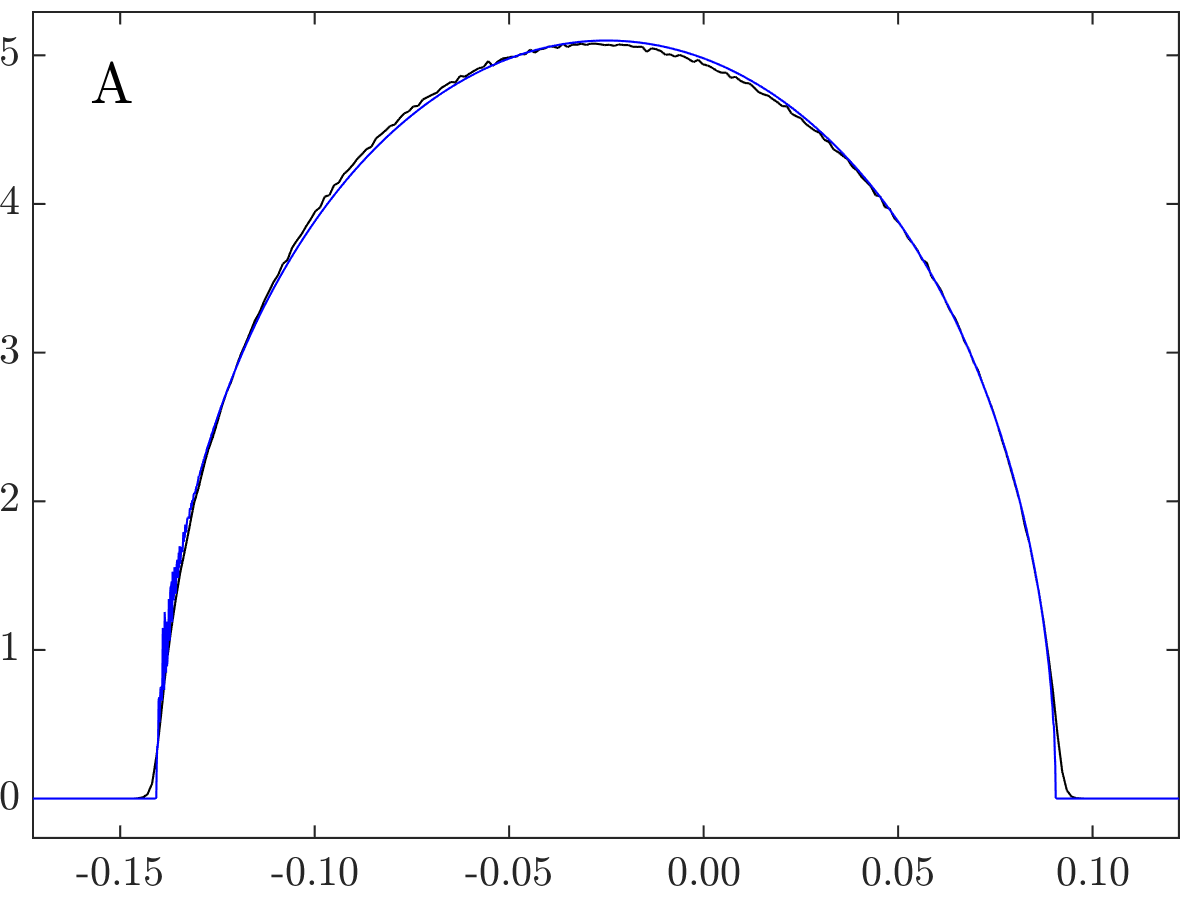} & \includegraphics[width=.2\textwidth]{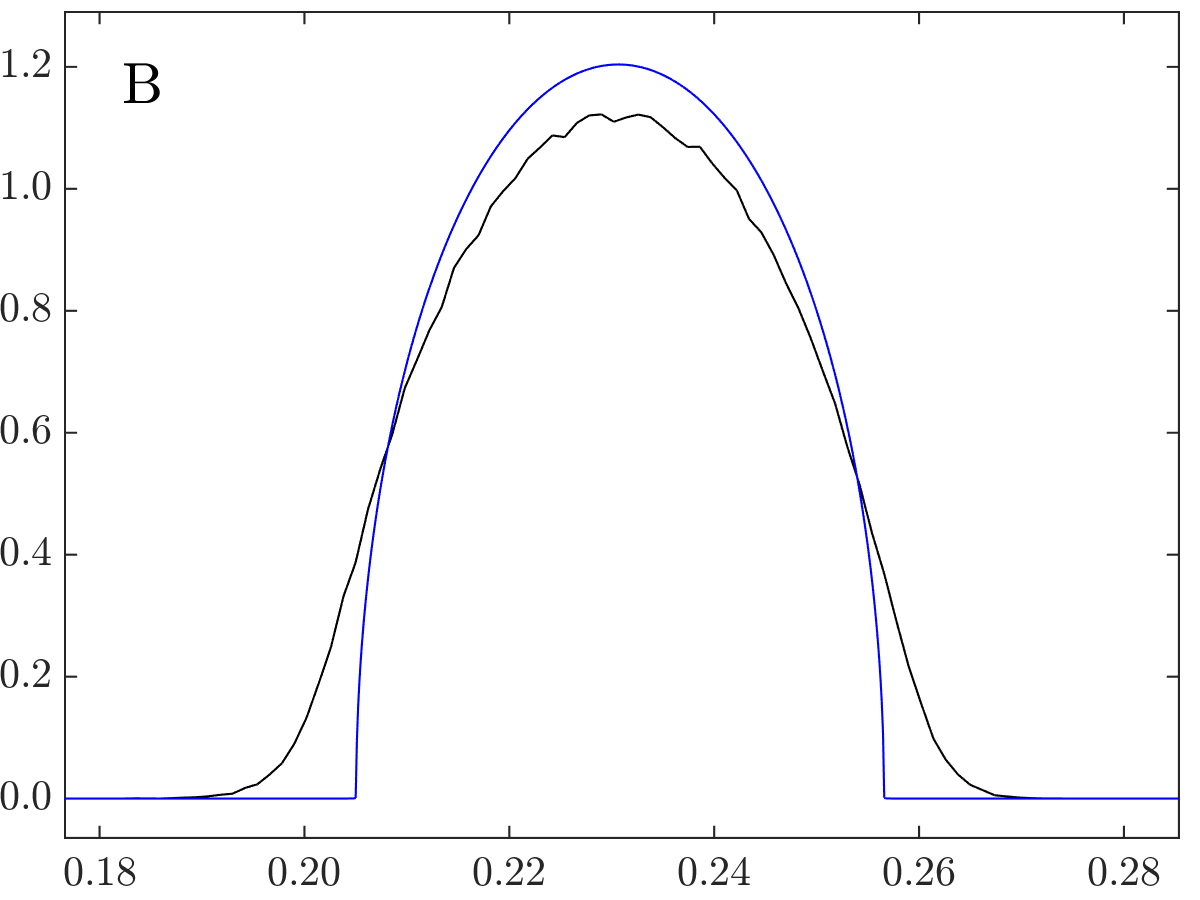} & \includegraphics[width=.2\textwidth]{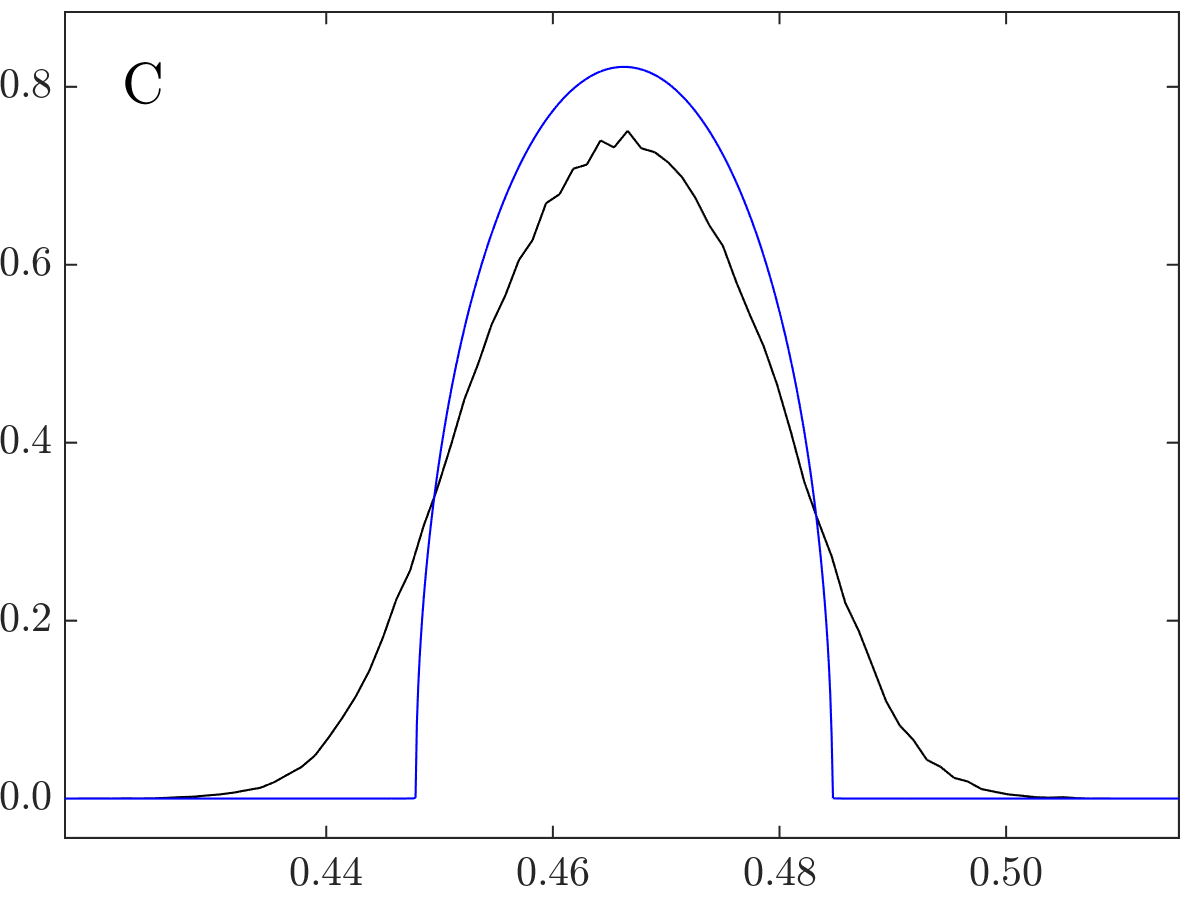} & \includegraphics[width=.2\textwidth]{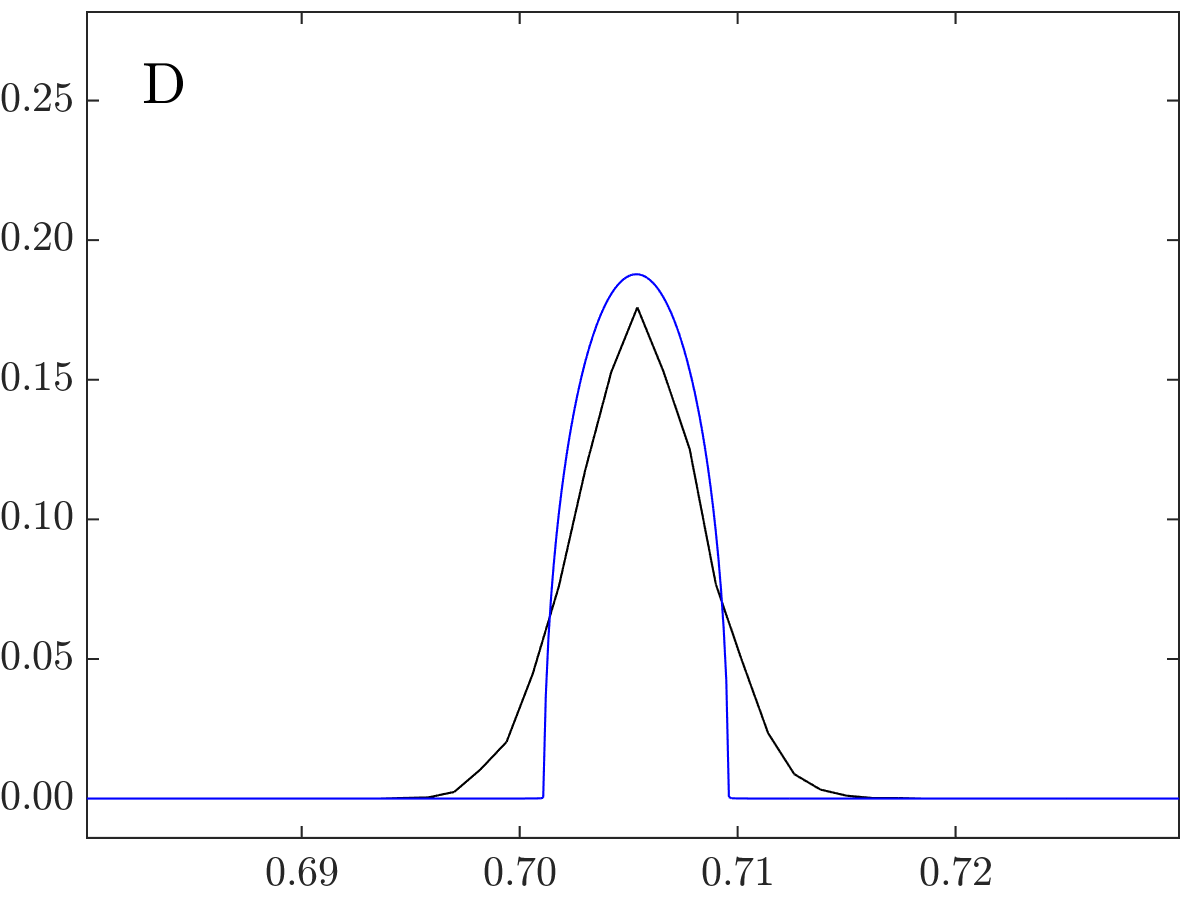}
\end{tabular}
\caption{For a graph percolation model $\mathcal{G}_\mathrm{perc}$ based on a two dimensional lattice with size $20~\mathrm{ nodes }\times 40~\mathrm{ nodes }$ and link inclusion probability $p=.7$, the expected empirical spectral distribution $\operatorname{E}\left[F_{W_N}\right]$ of the scaled adjacency matrix $W_N=\frac{1}{\gamma}A\left(\mathcal{G}_\mathrm{perc}\right)$ was computed over $10^4$ Monte-Carlo trials (black dotted line, top left).  The deterministic distribution $F_N$ computed through Corollary \ref{MainCor} was also computed (blue line, top left).  The corresponding density functions appear on the top right with labeled sections of interest more closely examined on the bottom row.}
\label{FigEx1}
\end{figure*}

\begin{figure*}[t]
\centering
\begin{tabular}{cccc}
\multicolumn{2}{c}{\includegraphics[width=.4\textwidth]{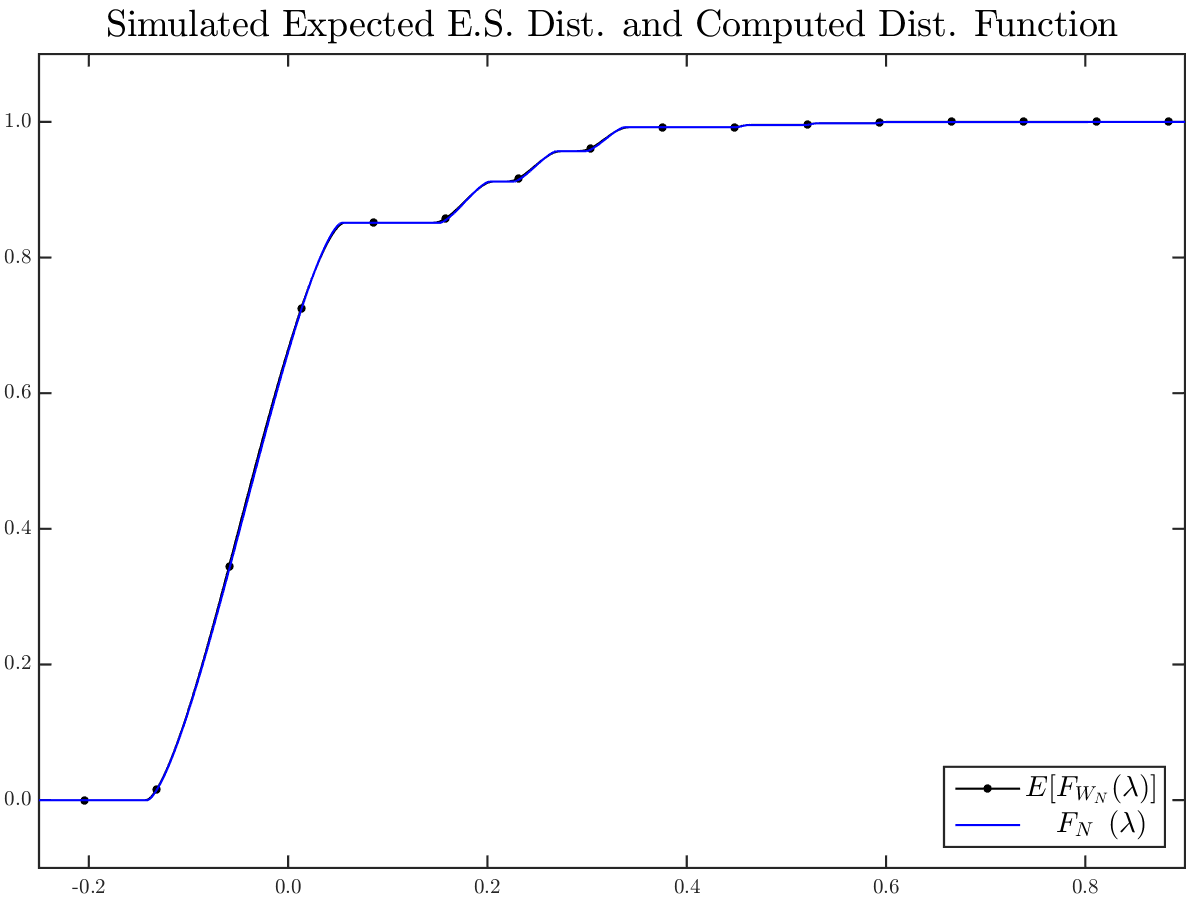}} & \multicolumn{2}{c}{\includegraphics[width=.4\textwidth]{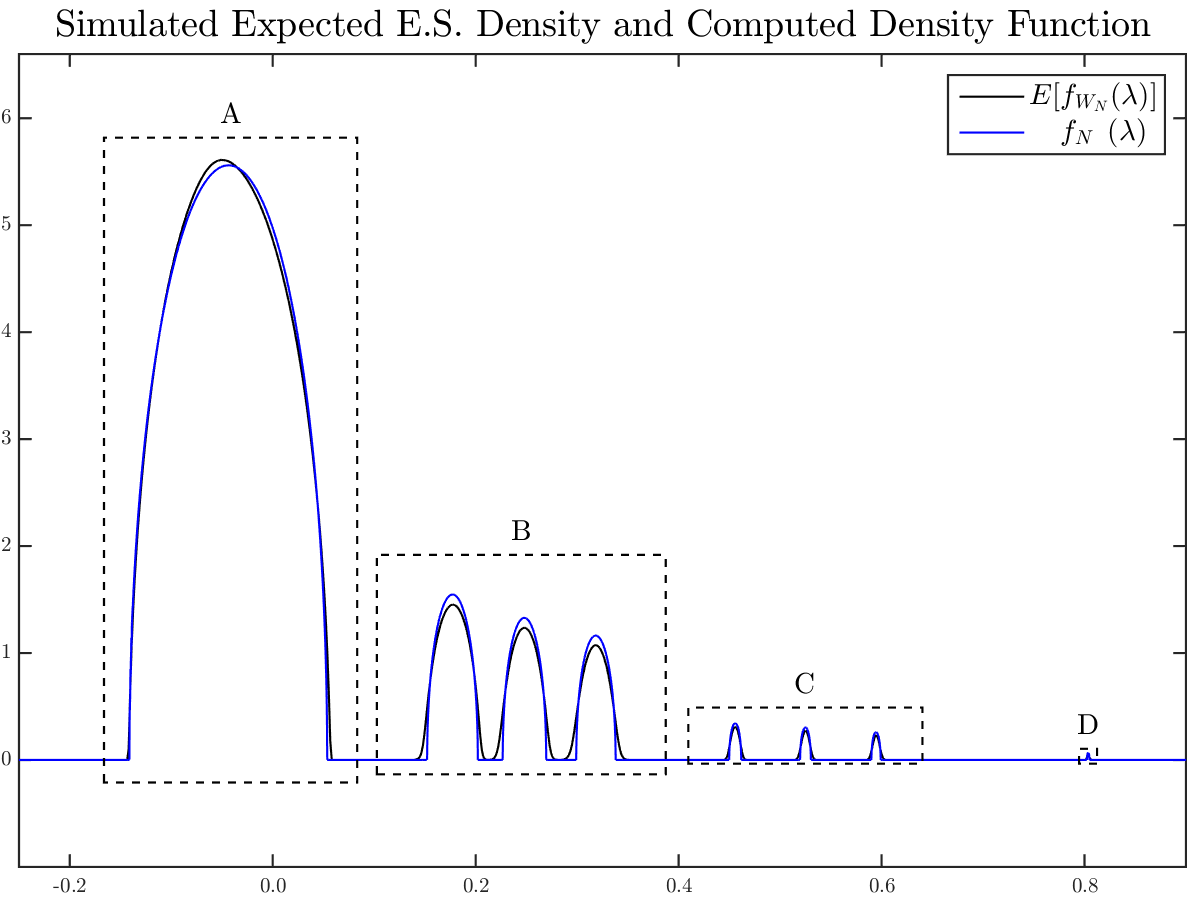}}\\
\includegraphics[width=.2\textwidth]{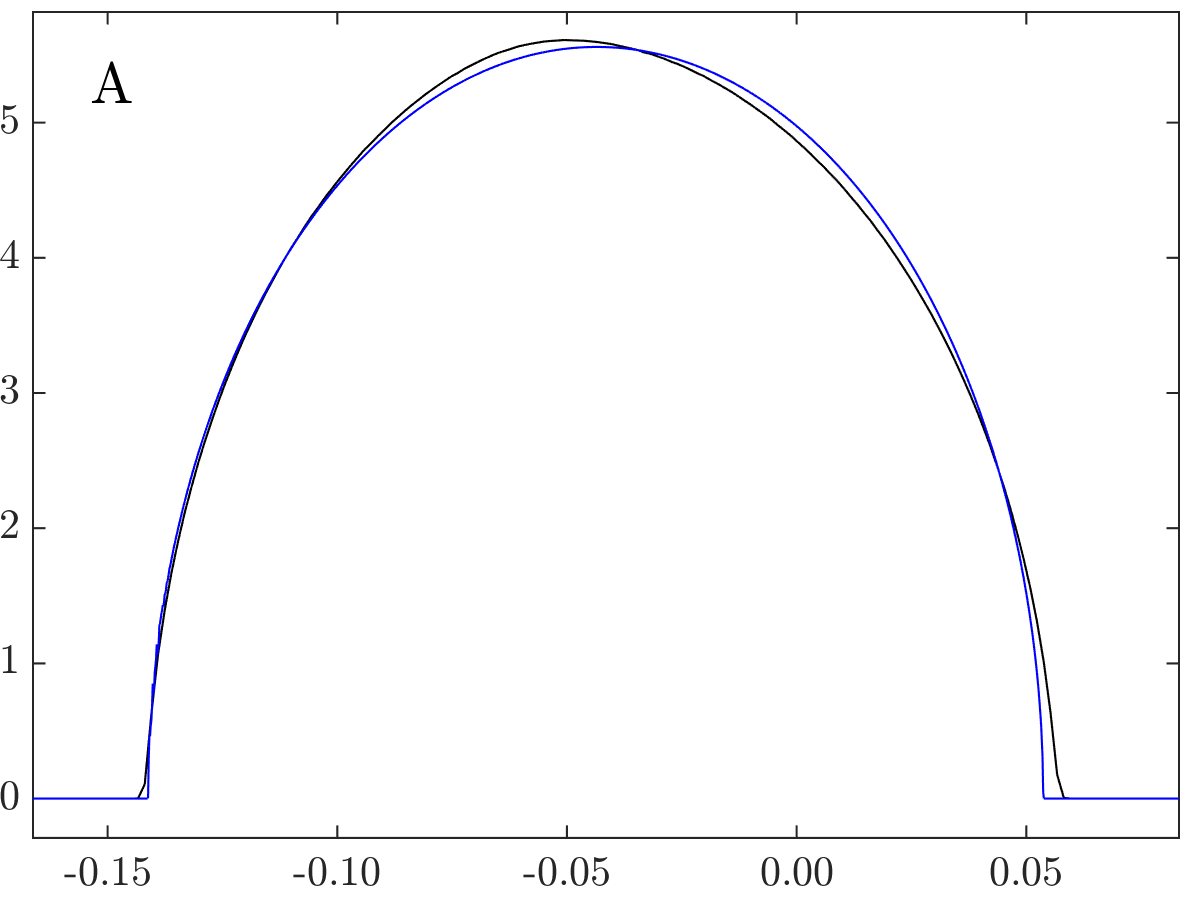} & \includegraphics[width=.2\textwidth]{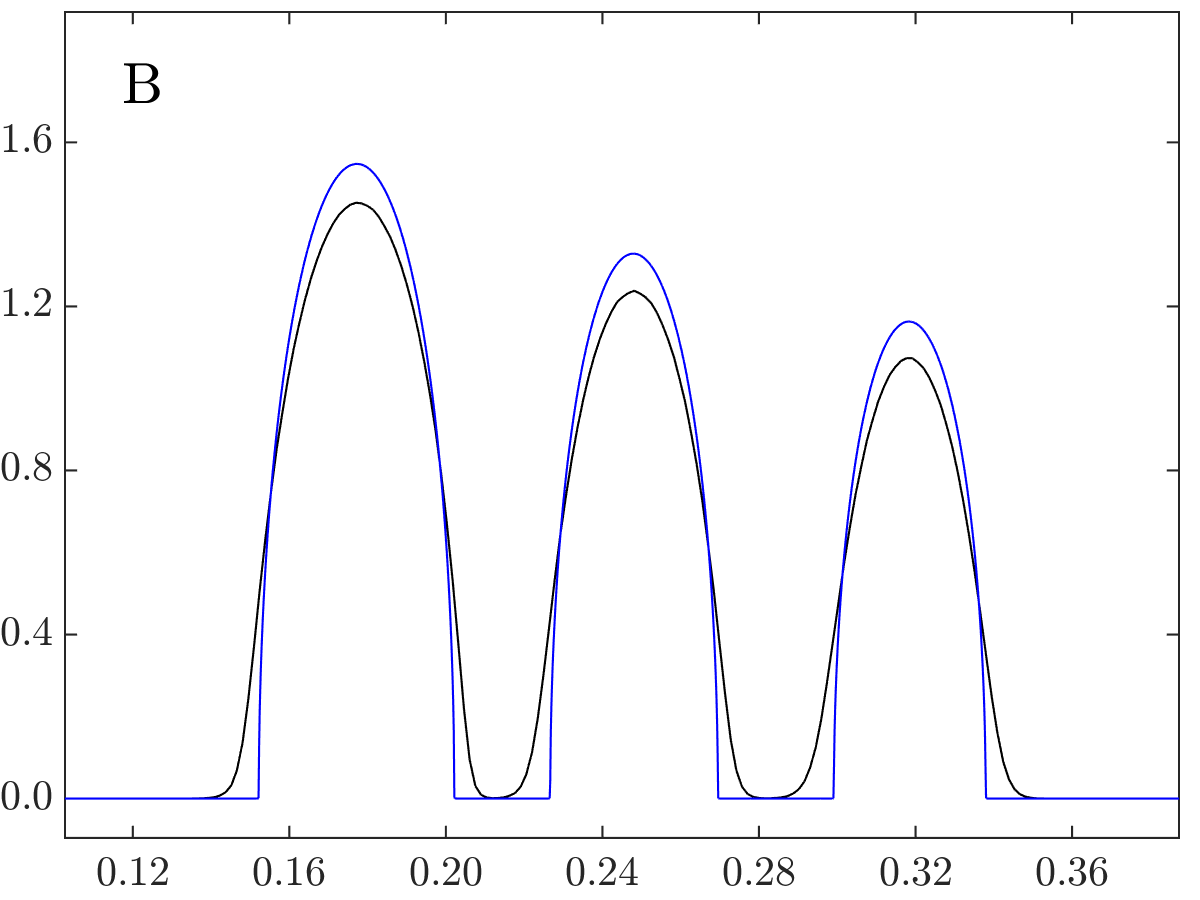} & \includegraphics[width=.2\textwidth]{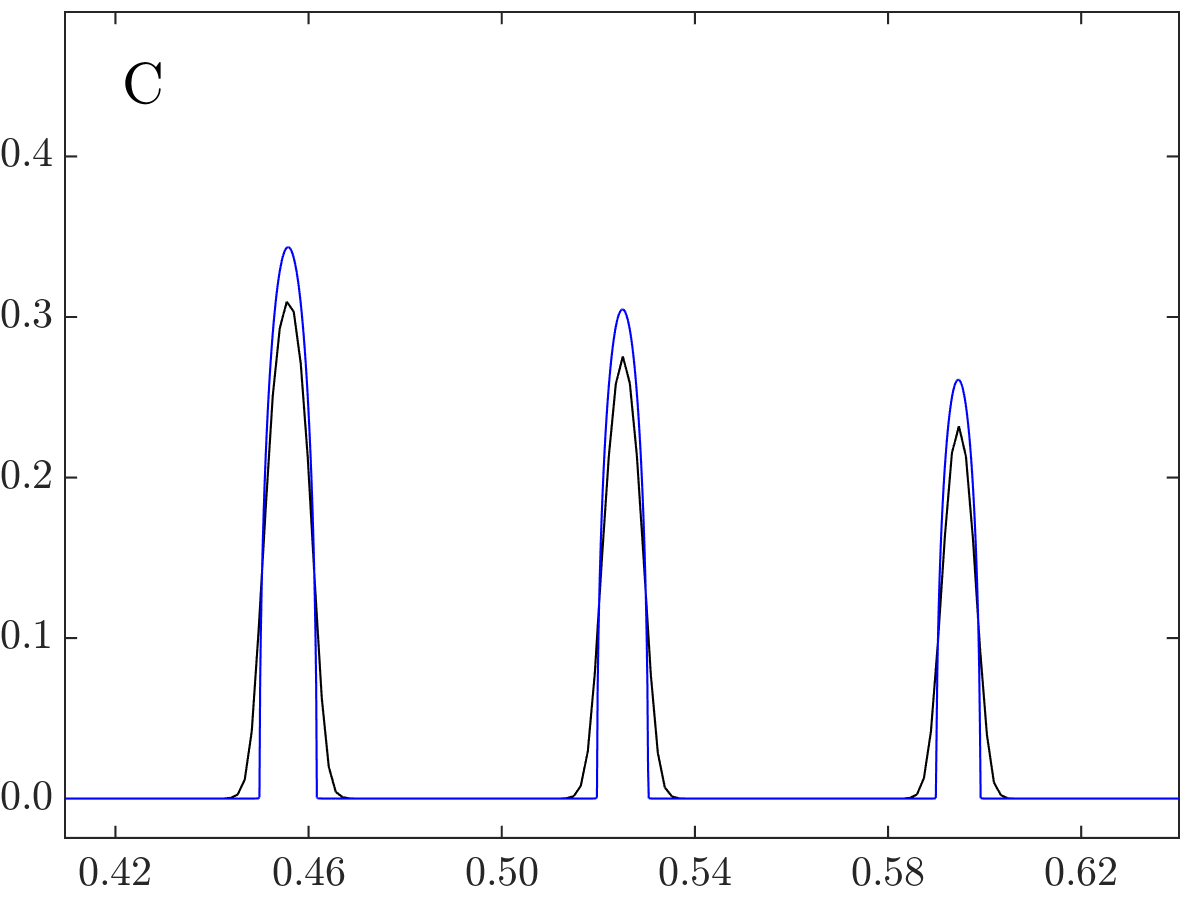} & \includegraphics[width=.2\textwidth]{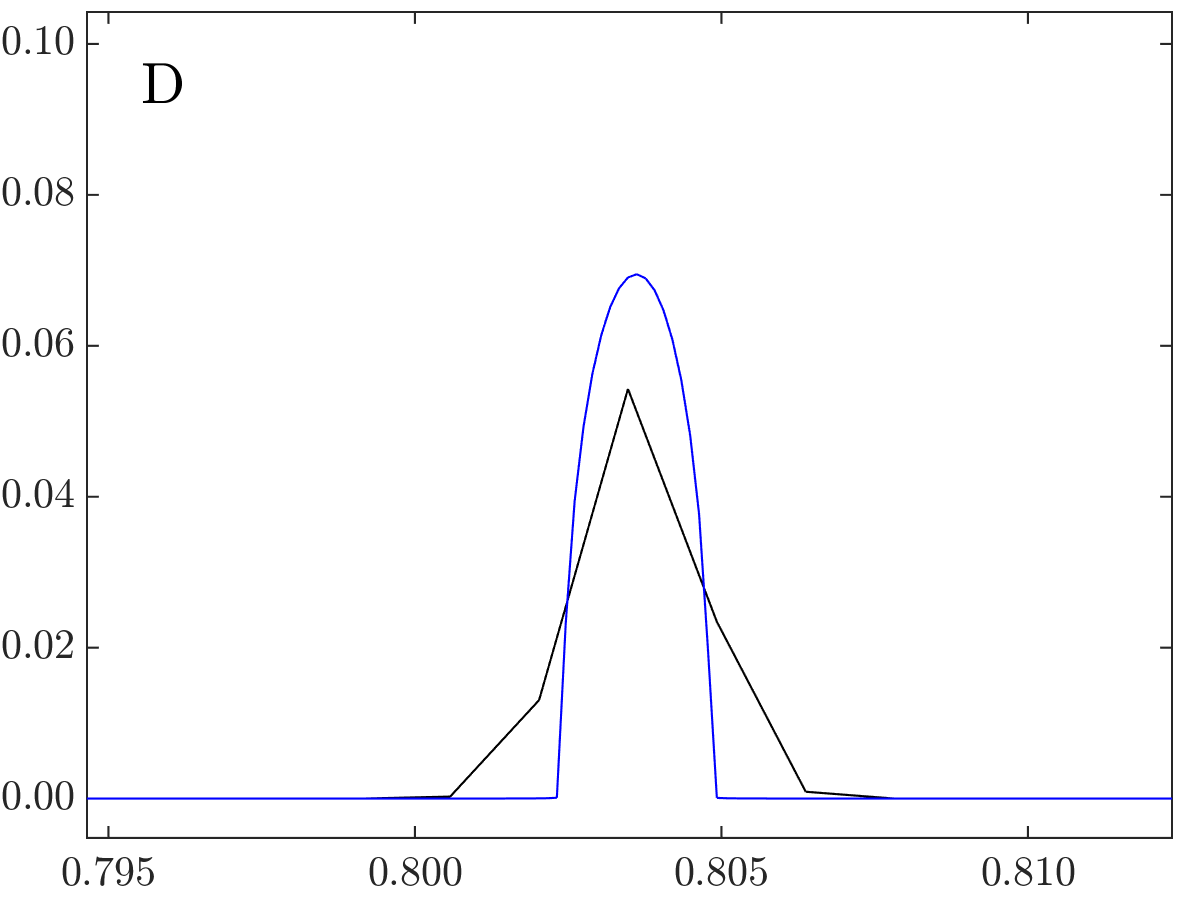}
\end{tabular}
\caption{For a graph percolation model $\mathcal{G}_\mathrm{perc}$ based on a two dimensional lattice with size $15 ~\mathrm{ nodes } \times 20~\mathrm{ nodes }\times 25~\mathrm{ nodes }$ and link inclusion probability $p=.8$, the expected empirical spectral distribution $\operatorname{E}\left[F_{W_N}\right]$ of the scaled adjacency matrix $W_N=\frac{1}{\gamma}A\left(\mathcal{G}_\mathrm{perc}\right)$ was computed over $10^4$ Monte-Carlo trials (black dotted line, top left).  The deterministic distribution $F_N$ computed through Corollary \ref{MainCor} was also computed (blue line, top left).  The corresponding density functions appear on the top right with labeled sections of interest more closely examined on the bottom row.}
\label{FigEx2}
\end{figure*}

\begin{remark}[Iterative Computation of Coefficients]\label{IterRmk}
Computation of the values of $\alpha_{i_1,\ldots,i_D}$ for $i_1,\ldots,i_D=0,1$ can be approached iteratively.  Let the $2^D$ dimensional vector $\boldsymbol{\alpha}$ collect the values of $\alpha_{i_1,\ldots,i_D}$ for $i_1,\ldots,i_D=0,1$.  Then the system of equations in \eqref{CoeffSystem} can be written as $X\boldsymbol{\alpha}=g\left(\boldsymbol{\alpha}\right)$ where $X$ is the invertible matrix describing the left side of the  equation and $g$ is the function describing the right side of the equation.  Selecting an initial guess $\boldsymbol{\alpha}_0$ and iterating such that $\boldsymbol{\alpha}_{n+1}=X^{-1}g\left(\boldsymbol{\alpha}_n\right)$ is a possible approach to search for the fixed point $\boldsymbol{\alpha}$.  While convergence of this process is not proven, it appears to work well in practice.

\end{remark}

Figure \ref{FigEx1} and Figure \ref{FigEx2} provide example deterministic equivalents for the empirical spectral distribution of scaled adjacency matrices for percolation models of lattice graphs with different lattice characteristics and link inclusion parameters as calculated using Corollary \ref{MainCor} and the methods in Remark \ref{IterRmk}.  The corresponding density function is also computed, and the expected empirical spectral distributions and densities are shown for comparison.  Asymptotically, as each lattice dimension size grows without bound, the area of the largest region of the density function approaches 1, and the area of the other regions approach 0.  Thus, Theorem \ref{GirkoK01Thm} does not guarantee the deterministic distribution function reflects the empirical spectral distribution in the regions of smaller increase.  Nevertheless, the observed characteristics seem to provide a good approximation.

Finally, Theorem \ref{NormalizedThm} describes the relationship between the empirical spectral distribution of the row-normalized adjacency matrix $\widehat{A}\left(\mathcal{G}_{\mathrm{perc}}\right)=\Delta^{-1}\left(\mathcal{G}_{\mathrm{perc}}\right)A\left(\mathcal{G}_{\mathrm{perc}}\right)$, where $\Delta\left(\mathcal{G}_{\mathrm{perc}}\right)$ is the diagonal matrix of node degrees, and the empirical spectral distribution of the scaled adjacency matrix $W\left(\mathcal{G}_{\mathrm{perc}}\right)=\frac{1}{\gamma}A\left(\mathcal{G}_{\mathrm{perc}}\right)$.  Because the row-normalized adjacency matrix and the symmetrically normalized adjacancy matrix are similar, the result also applies to both of those matrices.  With some manipulation, it can also be used to describe the row-normalized Laplacian and the symmetrically normalized Laplacian.  The theorem shows that the L\'{e}vy metric distance $d_{\scriptscriptstyle\mathrm{L}}\left(F_{p\sqrt{\gamma}\widehat{A}},F_{\sqrt{\gamma}W}\right)$ between $F_{p\sqrt{\gamma}\widehat{A}}$ and $F_{\sqrt{\gamma}W}$ approaches $0$ asymptotically.  The proof follows that of Lemma 5 of \cite{KAvr1} almost identically, with only some complications in applying Lemma 2.3 of \cite{CBor1} to bound the matrix entry differences.  Consequently, the deterministic equivalent $F_N$ computed in Corollary \ref{MainCor} can be used to approximate the empirical spectral distribution of $\widehat{A}$ as well.

\begin{theorem}\label{NormalizedThm}

Let $W\left(\mathcal{G}_{\mathrm{perc}}\right)=\frac{1}{\gamma}A\left(\mathcal{G}_{\mathrm{perc}}\right)$ be the scaled adjacency matrix of $\mathcal{G}_{\mathrm{perc}}\left(\mathcal{G}_{\mathrm{lat}},p\right)$ for scale factor $\gamma=\sum_{d=1}^{D}\left(M_d-1\right)$ with empirical spectral distribution $F_{W}$, and let $\widehat{A}\left(\mathcal{G}_{\mathrm{perc}}\right)=\Delta^{-1}\left(\mathcal{G}_{\mathrm{perc}}\right)A\left(\mathcal{G}_{\mathrm{perc}}\right)$ be the row-normalized adjacency matrix of $\mathcal{G}_{\mathrm{perc}}\left(\mathcal{G}_{\mathrm{lat}},p\right)$  with empirical spectral distribution $F_{\widehat{A}}$, where $\Delta\left(\mathcal{G}_{\mathrm{perc}}\right)$ is the diagonal matrix of node degress.  Also let $d_{\scriptscriptstyle\mathrm{L}}\left(\cdot,\cdot\right)$ be the L\'{e}vy distance metric.  Assume that all of the lattice dimension sizes increase without bound as $N\rightarrow \infty$.  Then,
\begin{equation}
\lim_{N\rightarrow\infty}{d_{\scriptscriptstyle\mathrm{L}}\left(F_{p\sqrt{\gamma}\widehat{A}},F_{\sqrt{\gamma}W}\right)}=0.
\end{equation}

\end{theorem}

\begin{proof}

Following the proof methods of Lemma 5 of \cite{KAvr1}, the row sums $\Delta_{ii}$ are first computed as $\gamma\left(p+\epsilon_i\right)$ where $\epsilon=\max_{i}\left|\epsilon_i\right|=\oo\left(1\right)$.  Subsequently, the bound on the cubed L\'{e}vy distance found in Lemma 2.3 of \cite{ZBai1} is evaluated and shown to asymptotically approach $0$.  Consequently, the result holds.  The first of these steps is less straightforward than the equivalent in \cite{KAvr1}.

Let $\beta\left(x,d\right)$ be defined as before.  For a given node index $i$ and lattice dimension $d$, form a $M_d \times M_d$ matrix $Q\left(i,d\right)$ from all entries $A_{lj}$ such that $\beta\left(l,k\right)=\beta\left(j,k\right)=\beta\left(i,k\right)$ for $k\neq d$.  Note that this matrix is symmetric with zeros on the diagonal and identically distributed Bernoulli random variables with parameter $p$ for all entries off the diagonal.  These variables are independent, except for symmetry relations.  Form a $M_d \times M_d$ matrix $R\left(i,d\right)$ equal to $Q\left(i,d\right)$ on the off diagonal but with an identical Bernoulli random variable added to the diagonal element.  By Lemma 2.3 of \cite{CBor1},
\begingroup
\thinmuskip=\muexpr\thinmuskip*1/16\relax
\medmuskip=\muexpr\medmuskip*1/16\relax
\begin{equation}
\begin{aligned}
\sum_{s=1}^{M_d}R_{l,s}\left(i,d\right)&=M_d\left(p+\delta_{d,l}\right)\\
&=\left(M_d-1\right)\left(p+\delta_{i,d,l}\right)+\left(p+\delta_{i,d,l}\right)
\end{aligned}
\end{equation}
\endgroup
where $\max_{l}\left|\delta_{i,d,l}\right|=\oo(1)$.
Thus, note that
\begin{equation}
\sum_{s=1}^{M_d}Q_{l,s}\left(i,d\right)=\left(M_d-1\right)\left(p+\kappa_{i,d,l}\right)
\end{equation}
where $\left(p+\delta_{i,d,l}-1\right)/M_d \leq \kappa_{i,d,l} \leq \left(p+\delta_{i,d,l}\right)/M_d $ so $\max_{l}\left|\kappa_{i,d,l}\right|=\oo(1)$.
Hence, $\Delta_{ii}$ may be computed as
\begin{equation}
\begin{aligned}
\Delta_{ii}=\sum_{j=1}^{N}{A_{ij}}
&=\sum_{d=1}^{D}{\sum_{s=1}^{M_d}{Q_{\beta\left(i,d\right),s}\left(i,d\right)}}\\
&=\sum_{d=1}^{D}{\left(M_d-1\right)\left(p+\kappa_{i,d,\beta\left(i,d\right)}\right)}\\
&=\gamma\left(p+\epsilon_i\right)
\end{aligned}
\end{equation}
where
\begin{equation}
\epsilon_i=\frac{1}{\gamma}\sum_{d=1}^{D}{\left(M_d-1\right)\kappa_{i,d,\beta\left(i,d\right)}}.
\end{equation}
Note that
\begin{equation}
\max_i\left|\epsilon_i\right|\leq \max_{i}{\max_d{\max_l{\left|\kappa_{i,d,l}\right|}}}
\end{equation}
so $\epsilon=\max_i\left|\epsilon_i\right|=\oo\left(1\right)$.
By Lemma 2.3 of \cite{ZBai1},
\begin{equation}
d_{\scriptscriptstyle\mathrm{L}}\left(F_{p\sqrt{\gamma}\widehat{A}},F_{\sqrt{\gamma}W}\right)\leq \frac{1}{n}\left\|\sqrt{\gamma}W-p\sqrt{\gamma}\widehat{A}\right\|_{F}^2.
\end{equation}
Note that $W_{ij}=\frac{1}{\gamma}A_{ij}$ and that
\begin{equation}
\begin{aligned}
p\widehat{A}_{ij}&=\frac{p}{\Delta_{ii}}A_{ij}\\
&=\frac{p}{\gamma\left(p+\epsilon_i\right)}A_{ij}\\
&=\frac{1}{\gamma}A_{ij}\left(1+\OO\left(\epsilon_i\right)\right)\\
&=\frac{1}{\gamma}A_{ij}\left(1+\OO\left(\epsilon\right)\right)
\end{aligned}
\end{equation}
Hence,
\begin{equation}
\begin{aligned}
\left\|\sqrt{\gamma}W-p\sqrt{\gamma}\widehat{A}\right\|_{F}^2
&=\frac{\gamma}{n}\sum_{i,j=1}^{N}\frac{1}{\gamma^2}A_{ij}^2\OO\left(\epsilon^2\right)\\
&=\frac{1}{n\gamma}\sum_{i,j=1}^{N}A_{ij}^2\OO\left(\epsilon^2\right)\rightarrow 0
\end{aligned}
\end{equation}
because $\epsilon=\oo\left(1\right)$.  Therefore, it has been shown that $d_{\scriptscriptstyle\mathrm{L}}\left(F_{p\sqrt{\gamma}\widehat{A}},F_{\sqrt{\gamma}W}\right)\rightarrow 0$.\hfill$\blacksquare$
\end{proof}

\section{Conclusion}\label{SecConclusion}

The eigenvalues of matrices that encode network information, such as the adjacency matrix and the Laplacian, provide important information that can be used, for instance, in polynomial filter design for graph signal processing.  This paper examines the empirical spectral distribution of the scaled adjacency matrix eigenvalues for a random graph model formed by independently including or excluding each link of a $D$-dimensional lattice supergraph, a percolation model.  A stochastic canonical equations theorem provided by Girko can be applied to find a sequence of deterministic functions that asymptotically approximates the empirical spectral distribution as the number of nodes increases.  Through this theorem, the Stieltjes transforms of these deterministic distributions can be found by solving a certain system of equations.  The primary contributions of this paper focus on solving this system of equations for the random network model studied with arbitrary lattice and link inclusion parameters.  Specifically, Theorem \ref{MainThm} finds the form of the matrix that satisfies the system of equations, resulting in a matrix with $2^D$ parameters.  Subsequently, Corollary \ref{MainCor} derives a system of $2^D$ equations that describe these parameters, significantly reducing the original system of equations.  The resulting distributions can also be related to the empirical spectral distribution of the symmetrically normalized adjacency matrix, row normalized adjacency matrix, symmetrically normalized Laplacian, and row-normalized Laplacian.  Simulations for sample model parameters validate the results.  While the methods used guarantee that the computed deterministic distribution function asymptotically captures the behavior of the bulk of the eigenvalues, no guarantees are made about regions where the fraction of eigenvalues asymptotically vanishes.  Although the observed correspondence between the computed and simulated distributions appears good, future efforts could include a more precise characterization.  Additional topics for possible expansion could also include evaluation of the degree to which filter design benefits from the distribution information gained as well as analysis of additional random network models of interest.

\end{document}